\pdfoutput=1
\RequirePackage{ifpdf}
\ifpdf 
\documentclass[pdftex]{sigma}
\else
\documentclass{sigma}
\fi

\numberwithin{equation}{section}

\newtheorem{Theorem}{Theorem}[section]
\newtheorem{Corollary}[Theorem]{Corollary}
\newtheorem{Lemma}[Theorem]{Lemma}
\newtheorem{Proposition}[Theorem]{Proposition}
\newtheorem{Conjecture}[Theorem]{Conjecture}
 { \theoremstyle{definition}

\newtheorem{Example}[Theorem]{Example}
 }

\usepackage[russian,english]{babel}
\usepackage{braket}
\usepackage{bm}
\usepackage{tikz-cd}
\NewDocumentCommand\pr{m}{\ensuremath{\left(#1\right)}}
\NewDocumentCommand\of{m}{\ensuremath{\!\pr{#1}}}

\renewcommand{\C}{\mathbb{C}}
\newcommand{\R}{\mathbb{R}}
\newcommand{\Q}{\mathbb{Q}}
\newcommand{\Z}{\mathbb{Z}}
\newcommand{\OO}{\mathcal{O}}
\renewcommand{\Proj}[1]{{\mathbb{P}^{#1}}}

\NewDocumentCommand\rank{m}{\ensuremath{\operatorname{rank}{#1}}}
\NewDocumentCommand\GL{m}{\ensuremath{\operatorname{GL}\of{#1}}}
\NewDocumentCommand\Lie{mo}{\ensuremath{\mathfrak{\MakeLowercase{#1}}\IfValueT{#2}{\of{#2}}}}
\NewDocumentCommand\PSL{m}{\ensuremath{\mathbb{P}\!\operatorname{SL}_{#1}}}
\NewDocumentCommand\nForms{omm}
{%
\IfValueTF{#1}
{\ensuremath{\vb{#1}^{#2}}}
{\ensuremath{\Omega^{#2}_{#3}}}
}%
\NewDocumentCommand\cohomology{omm}{\ensuremath{H_{\IfValueT{#1}{\IfStrEq{#1}{db}{\overline{\partial}}{#1}}}^{#2}\of{#3}}}
\DeclareMathOperator{\Ad}{Ad}
\DeclareMathOperator{\ad}{ad}
\NewDocumentCommand\LieDer{}{\ensuremath{\mathcal L}}
\NewDocumentCommand\homotopyGroup{mm}{\ensuremath{\pi_{#1}({#2})}}
\NewDocumentCommand\fundamentalGroup{m}{\ensuremath{\homotopyGroup{1}{#1}}}
\NewDocumentCommand\MakeLie{m}{\expandafter\def\csname Lie#1\endcsname{\Lie{#1}}}
\NewDocumentCommand\graded{mm}{\ensuremath{\operatorname{gr}_{#1}{#2}}}
\NewDocumentCommand\prodquot{mmm}{\ensuremath{#1 \times^{#3}\! #2}}
\newcommand{\GZ}{G_0}
\newcommand{\tGZ}{\widetilde{G}_0}
\newcommand{\LieGZ}{\LieG_0}
\makeatletter
\NewDocumentCommand\cb{O{1}m}%
{\ensuremath{%
\ifnum\pdf@strcmp{#1}{0}=\z@%
 {\OO{}_{#2}}%
 \else{
 \ifnum\pdf@strcmp{#1}{1}=\z@%
 {}%
 \else{
 \ifnum\pdf@strcmp{#1}{-}=\z@%
 {-}%
 \else{
 \ifnum\pdf@strcmp{#1}{-1}=\z@%
 {-}%
 \else%
 {#1}%
 \fi
 }%
 \fi
 }%
 \fi
 K_{#2}
 }
\fi
}
}
\makeatother
\NewDocumentCommand\acb{O{1}m}{\newcount\n \n=#1 \multiply\n by -1 \cb[\number\n]{#2}}
\newcommand{\amal}[3]{\ensuremath{{#1} \mathbin{\times^{#2}} \! #3}}
\makeatletter
\newcommand*{\Shaf}[2][.]%
{
\operatorname{\text{\foreignlanguage{russian}{ш}}}\ifnum\pdf@strcmp{#1}{.}=\z@\else_{#1}\fi\ifnum\pdf@strcmp{#2}{}=\z@\else\! #2\fi
}
\newcommand*{\Ii}[2][.]%
{
\operatorname{Ii}\ifnum\pdf@strcmp{#1}{.}=\z@\else_{#1}\fi#2
}
\makeatother
\def\lst{B,G,H,K,L,P,Q,S,Z}
\makeatletter
\@for\i:=\lst\do{\expandafter\MakeLie \i}
\makeatother
\newcommand*{\dimC}[1]{\ensuremath{\dim_{\C} \! #1}}
\NewDocumentCommand\vb{m}{{\bm{#1}}}
\NewDocumentCommand\vbTM{}{\vb{{\mathfrak{g}/{\mathfrak{h}}}}}
\NewDocumentCommand\Iitaka{O{\cb{}}m}{\ensuremath{{}^{#1}\!{#2}}}
\newcommand*{\acf}[1]{\Iitaka[\acb{}]{#1}}
\NewDocumentCommand\red{m}{#1_{\rm r}}
\NewDocumentCommand\urad{m}{#1_{\rm u}}

\RenewDocumentCommand\to{}{\,\longrightarrow\,}
\NewDocumentCommand\rationalMap{omm}{\ensuremath{\IfValueTF{#1}{#1 \colon #2\mathbin{\,\tikz \draw[dashed,densely dashed,->,line cap=round] (.1,0) -- (.65,0);\,} #3}{#2\mathbin{\,\tikz \draw[dashed,densely dashed,->,line cap=round] (.1,0) -- (.65,0);\,} #3}}}

\begin{document}

\allowdisplaybreaks

\newcommand{\arXivNumber}{2302.13649}

\renewcommand{\PaperNumber}{030}

\FirstPageHeading

\ShortArticleName{Locally Homogeneous Holomorphic Geometric Structures on Projective Varieties}

\ArticleName{Locally Homogeneous Holomorphic Geometric\\ Structures on Projective Varieties}

\Author{Indranil BISWAS~$^{\rm a}$ and Benjamin MCKAY~$^{\rm b}$}
\AuthorNameForHeading{I.~Biswas and B.~McKay}
\Address{$^{\rm a)}$~Department of Mathematics, Shiv Nadar University,\\
\hphantom{$^{\rm a)}$}~NH91, Tehsil Dadri, Greater Noida, Uttar Pradesh 201314, India}
\EmailD{\href{mailto:indranil.biswas@snu.edu.in}{indranil.biswas@snu.edu.in}}
\URLaddressD{\url{https://snu.edu.in/faculty/indranil-biswas/}}

\Address{$^{\rm b)}$~School of Mathematical Sciences, University College Cork, Cork, Ireland}
\EmailD{\href{mailto:b.mckay@ucc.ie}{b.mckay@ucc.ie}}
\URLaddressD{\url{https://ben-mckay.github.io/benmckay.github.io/}}

\ArticleDates{Received April 01, 2023, in final form March 29, 2024; Published online April 08, 2024}

\Abstract{For any smooth projective variety with holomorphic locally homogeneous structure modelled on a homogeneous algebraic variety, we determine all the subvarieties of it which develop to the model.}


\Classification{53B21; 53C56; 53A55}

\section{Introduction}

This paper is concerned with the following classification problems:
\begin{itemize}\itemsep=0pt
\item
the complex manifolds admitting holomorphic Cartan geometries,
\item
the geometries they admit,
\item
the subvarieties of such a manifold which can be developed to the model of such a geometry,
\item
the symmetries of those geometries.
\end{itemize}
While none of these problems can be currently solved in complete generality, we solve various of them under additional hypotheses, for example requiring the complex manifold to be a smooth projective variety, or requiring the geometry to be parabolic or flat or both.
Recall that a~holomorphic locally homogeneous $(X, G)$-structure is precisely a flat holomorphic $(X, G)$-Cartan geometry.
Making use of terms defined below, we will prove:

\begin{Theorem}\label{theorem:main}
Suppose that
\begin{enumerate}\itemsep=0pt
\item[$(1)$]
$(X, G)$ is a complex algebraic homogeneous space,
\item[$(2)$]
$M$ is a connected smooth projective variety,
\item[$(3)$]
$M$ has a flat holomorphic $(X, G)$-Cartan geometry.
\end{enumerate}
Then, after perhaps replacing $M$ by a finite unramified covering of it, $M$ belongs to a tower of holomorphic fibrations
\[M\to M'\to M''\to M'''\to S,\] where
\begin{itemize}\itemsep=0pt
\item
$M'$, $M''$, $M'''$, $S$ have no rational curves,
\item
$M\to M'$ is a holomorphic fiber bundle mapping with flag variety fibers,
\item
the Cartan geometry on $M$ is lifted from $M'$,
\item
the maps $M'\to M'' \to M'''\to S$ are holomorphic fibrations of abelian varieties,
\item
the composition $M'\to M''\to M'''$ is a trivial fibration $M' = M'''\times A$ for some abelian
variety $A$,
\item
a finite map from a connected compact complex analytic variety to $M$ develops to a map to the model $X$, after perhaps replacing by a finite unramified covering, just when it lies in a finite unramified
covering of a fiber of $M\to M''$,
\item
every fiber of $M\to M''$ develops to a map to the model $X$, after perhaps replacing by a~finite unramified covering,
\item
the fibers of $M'\to M''$ are bounded in dimension by the dimension of the fibers of the ant fibration $X \longrightarrow
\overline{X}$ defined below,
\item
$M'\to S$ has a holomorphic section,
\item
$S$ has ample canonical bundle.
\end{itemize}
\end{Theorem}

In this theorem, we allow a single point to be considered a zero-dimensional flag variety, and also to be considered a zero-dimensional abelian variety, and also to be considered a zero-dimensional smooth projective variety with ample canonical bundle.

\begin{Example}
A \emph{flag variety} is a rational homogeneous projective variety \cite{Alekseevsky:1997}.
As we will see, holomorphic $(X, G)$-structures modelled on flag varieties (also known as \emph{parabolic geometries}) have
ant fibration $X \longrightarrow \overline{X} = X$ which coincides with the identity map.
Take any smooth projective variety $M$ with a holomorphic $(X, G)$-geometry.
Possibly replacing $M$ by a finite unramified covering space, we
have $M' = M''$, i.e., each abelian variety fiber is a single point.
Hence the fibers of the map $M\to M'$ are precisely the subvarieties which, after finite unramified covering, develop to the model.
These are flag varieties, more precisely they are the fibers~$X_0$ of a $G$-equivariant map $X_0\to X\to X'$ of flag varieties.
So if $M$ is not a fiber bundle with fibers $X_0$, then only points of $M$ develop to the model.
If $M$ is such a fiber bundle, then its fibers all develop to the model, precisely as fibers of $X_0\to X\to X'$.
This completes the classification of developing subvarieties of flat holomorphic parabolic geometries. In particular, if $X$ is
a minimal flag variety, i.e., $X = G/P$ for $P \subseteq G$ a maximal parabolic subgroup, then only points can develop from
any holomorphic $(X, G)$-geometry on any smooth projective variety, except if $M = X$, which can only bear the standard flat
model geometry, and develops to itself by any element of $G$. From this point of view, Theorem~\ref{theorem:main} is a no go
theorem: there are no additional important examples where we can apply developing maps, after our previous work~\cite{Biswas.McKay:2016} developing rational curves.
\end{Example}

\begin{Example}
There is a unique family of smooth projective varieties, the Jahnke--Radloff manifolds, which admit holomorphic projective connections and
which are not ball quotients, tori or projective spaces \cite{Jahnke/Radloff:2015}. By our results, since the model
$(X, G) = (\Proj{n}, \PSL{n+1})$ has trivial ant fibration (mentioned in Theorem~\ref{theorem:main}), no positive-dimensional
subvariety of these mysterious Jahnke--Radloff varieties develops to the model. We do not see how to prove this directly from the direct
construction of Jahnke and Radloff \cite{Jahnke/Radloff:2015}. We do not even see how to prove this directly for ball quotients from their
explicit construction.
\end{Example}

\begin{Example}
The group $G := \PSL{2}$ has its usual action on ${\mathbb C}{\mathbb P}^1$, which induces an action on $\text{Sym}^2\big({\mathbb C}{\mathbb P}^1\big)$
(i.e., the space of quadrics $ax^2+bxy+cy^2 = 0$ on the projective line).
Let~${X := \Proj{2}-Q}$ be the complement of the smooth quadric $Q = \big(b^2 = 4ac\big)$, hence $X$ is the set of quadrics on the projective line, i.e., the set of unordered pairs of points of the projective line.
The anticanonical bundle of $\Proj{2}$ is positive, so the anticanonical fibration has only points as fibers, therefore the same is
true of $X \subseteq \Proj{2}$, and consequently the ant fibration of $X$ also has only points as fibers.
Nothing is known about $(X, G)$-structures on complex surfaces.
Suppose that $M$ is a smooth projective surface with an $(X, G)$-structure.
Consider our tower
\[M \longrightarrow M' \longrightarrow M'' \longrightarrow M''' \longrightarrow S.\]
The first arrow $M \to M'$ has flag variety as fibers, and they develop to the model.
In the model, $X$ has no rational curves, since every rational curve in $\Proj{2}$ intersects the smooth quadric curve $Q$.
Flag varieties of positive dimension are covered by rational curves.
Hence the flag variety fibers of $M \to M'$ have dimension zero, i.e., $M' = M$.
Thus $M \to M''$ is an abelian fibration, and its fibers develop to $X$.
By the same argument, the fibers are points, and hence~${M = M' = M''}$.
Therefore, we have $M = M' = M'''\times A$.
If $M$ is an elliptic fibration, it has no singular fibers, and hence it is a product of a curve with an elliptic curve.
This means that $M$ is an abelian surface or a surface with ample canonical bundle or a product of an elliptic curve with some curve of
genus $ \ge 1$.

We can see more about this example from a different angle: since $X \subseteq \Proj{2}$, every $(X, G)$-structure imposes a
holomorphic flat projective connection. Thus $M$ bears a flat holomorphic projective connection, which
implies that it is the projective plane, a ball
quotient, or a torus with translation invariant flat holomorphic affine connection \cite[Main theorem]{KobayashiOchiai:1980}. Its
developing map to $\Proj{2}$ intersects $Q$ unless $M$ is a ball quotient. Each ball quotient has a unique holomorphic projective
connection, which is flat \cite{McKay:2016}, and hence it is an $(X, G)$-structure for this~$(X, G)$ since it develops to the ball in the
projective plane.
\end{Example}

Our research programme aims to classify holomorphic Cartan geometries on smooth projective varieties, as a natural class of Cartan
geometries not approachable by symmetry group methods, but perhaps approachable using recent advances in algebraic geometry. Our goal in
this paper is to prove Theorem~\ref{theorem:main}, which clearly provides strong constraints on the possible subvarieties in a flat
holomorphic Cartan geometry which can develop to the model. Roughly, it is a no go theorem: as we will see, it prevents the application of
developing maps in any of the important types of Cartan geometries. We expect that these results continue to hold for holomorphic Cartan
geometries which are not flat, as holomorphy is very similar to flatness~\cite{Atiyah:1957}. We also expect that these results continue to
hold on compact K\"ahler manifolds, and more generally on Fujiki manifolds (i.e., compact complex manifolds dominated by compact K\"ahler
manifolds), since the known examples are similar to the smooth projective examples.

It seems likely that most readers of this paper are differential geometers, interested in Cartan geometries, and not algebraic geometers,
so we provide complete definitions and detailed references for standard results in complex algebraic geometry, with apologies to the
algebraic geometers.

\section{The main ideas and techniques of the paper}

Throughout the paper, we imagine we are faced with a compact complex manifold $M$ and a~complex homogeneous space $X = G/H$, the model,
and that $M$ bears a flat holomorphic Cartan geometry modelled on $X$. We want to identify the possible choices of subvarieties~$Z \subseteq M$ which develop to $X$ by some developing map associated to the geometry; see Section~\ref{subsection:Development}
for a~precise definition. We approach the problem from two directions simultaneously: saying what we can about $Z$ in terms of $M$,
independent of choice of $X$, but also in terms of $X$, independent of choice of $M$.

Any flat holomorphic principal bundle over a connected compact complex manifold $M$ arises from a representation of its fundamental
group. Under some technical hypotheses, it is roughly true that this representation of the fundamental group factors through
$\pi_1(M) \to \pi_1(S)$ for some holomorphic map $M \to S$, the \emph{Shafarevich fibration}, with complex torus fibers, to some compact
complex manifold $S$, such that the map contracts just those subvarieties of $M$ on whose fundamental group the representation is
trivial, and hence $M\to S$ is essentially uniquely determined \cite{Kollar:1993,Zuo:1999}. (Actually, the map is not quite
defined, so there are some technical details.) Moreover $S$ has negative Ricci curvature in some K\"ahler metric, i.e., ample canonical
bundle.
The Shafarevich fibration we use is a special case of the collection of Shafarevich fibrations~$\operatorname{sh}_{\rho}$ of an algebraic variety \cite[p.~6]{Zuo:1999}, each associated to a representation $\rho$ of its fundamental group.
These $\operatorname{sh}_{\rho}$ all factor the universal Shafarevich fibration $\operatorname{sh}$ \cite[Definition 1]{Zuo:1999}, which does not depend on a representation.

Suppose that a subvariety $Z \subseteq M$ develops to $X$, so it is identified with a subvariety $Z \subseteq X$.
The holonomy of the Cartan connection must be trivial on $Z$ to avoid a multivalued developing map.
Therefore, $Z$ sits in a fiber of the Shafarevich fibration.
But conversely, since the holonomy is essentially trivial on the fibers of the Shafarevich fibration, and the Cartan geometry is
flat, the fibers themselves develop.
(This is not quite true: due to technical problems, the holonomy might only be solvable on the fibers, and hence the fibers might not develop.)

If a particular compact complex manifold $M$ were given explicitly, and we could somehow compute the Shafarevich fibration, we would
thus expect to know which subvarieties develop from~$M$, independent of $X$. But there are few examples of complex manifolds $M$
with holomorphic principal bundle and holomorphic flat connection for which the Shafarevich fibration is known. Again our aim is to say
what we can about $Z$ in terms of $M$, independent of choice of $X$, but also in terms of $X$, independent of choice of $M$.
Therefore, we want to have more information, that does not require knowing so much about both $M$ and $X$ simultaneously,
i.e., without knowing the Shafarevich fibration.

When we develop subvarieties, ambient tangent bundles are identified:
\[
 TM|_Z \cong TX|_Z
\]
by the developing map.
But $X$ has tangent bundle spanned by global holomorphic sections, since it is homogeneous.
In particular, $X$ has semipositive anticanonical bundle, i.e., lots of wedge products of holomorphic vector fields.

On the other hand, we saw that $M$ is essentially a bundle of tori over a manifold $S$ with negative Ricci curvature,
hence $M$ has tangent bundle something like a sum of a trivial bundle with some negative curvature directions.
In particular, the canonical bundle of $M$ is semipositive.
Therefore, along $Z$, the ambient canonical bundle is semipositive, from $M$, and seminegative, from $X$, hence trivial.
We can thus employ the huge theory of Calabi--Yau manifolds, i.e., smooth projective varieties with trivial canonical bundle.

This helps to identify the possible choices of $Z$ in terms of $M$, independent of model $X$, but also in terms of $X$,
independent of choice of $M$. We will see that this triviality forces $Z$ to be a~torus and to lie in a leaf of a certain foliation on
$M$, the \emph{ant foliation}, and also to lie in a~fiber of an associated fibration of $X$, the \emph{ant fibration}; see
Section~\ref{subsec:ant}. As we will see, for most of the homogeneous models $X$ of interest in geometry, the ant fibration has only
single points as fibers, and hence $Z$ is a point. Since we can compute the ant fibration of $X$ without knowing $M$, this gives explicit
examples of models $X$ for which nothing can develop from any $M$ without rational curves. Nonetheless, we provide examples in which
the ant fibration has arbitrarily large fibers, so our main theorem (see Theorem~\ref{theorem:main}) is difficult to state.

Our first attempt to produce such an invariant fibration, with the required relationship to the ambient canonical bundle, gives us a
previously known fibration called the \emph{anticanonical fibration}; see Section~\ref{section:the.anticanonical.splitting}. We refine
this fibration to a previously undiscovered invariant fibration of complex homogeneous manifolds, the \emph{ant fibration}, which we
demonstrate is, in some examples, strictly finer, even though it has essentially all the same theorems as the anticanonical fibration; see
Section~\ref{subsec:ant}. We believe the ant fibration is of independent interest.

Computing the dimension of fibers of the ant fibration of a homogeneous model $X$ gives an upper bound on the dimension of varieties
$Z$ which can develop from any Cartan geometry with that model $X$, independent of choice of $M$.

The ant fibration of the model induces, via the Cartan geometry, a foliation of any complex manifold $M$ with any holomorphic Cartan
geometry, the \emph{ant foliation}. If $M$ has no rational curves, then we saw $M = M_0\times A$. Every leaf of the ant foliation lies
inside a fiber $\text{pt}\times A$; see Corollary~\ref{cor:splitting}. Noting that $A$ is the identity component of the
biholomorphism group of $M$, we see that fewer symmetries of the complex manifold $M$ imply smaller dimensions of subvarieties~$Z$
which develop to the model $X$, independent of choice of $X$. In particular, ``most'' compact complex manifolds $M$ with no rational
curves will only allow individual points $Z$ to develop to the model of any holomorphic Cartan geometry on $M$.

\section{Summary of the paper}

In Section~\ref{section:review}, we explain briefly the concept of Cartan geometry. This paper concerns only flat Cartan geometries, but
various of its lemmata are independently useful and so are stated in the generality of holomorphic Cartan geometries. We are hopeful that
the results generalize to arbitrary holomorphic Cartan geometries. Flatness might seem an extreme hypothesis, but at the moment all smooth
projective varieties which are known to admit a holomorphic Cartan geometry are known to admit a flat one, and many of our results only
require that there be a~holomorphic Cartan geometry and also some holomorphic flat connection on the tractor bundle.

In Section~\ref{subsection:Development}, we define the notion, in any Cartan geometry, of development of submanifolds on which curvature vanishes.
This first appeared in our previous work, but is a minor generalization of various well known definitions, as we explain.

Section~\ref{section:algebraic.geometry} summarizes, with complete references to standard textbooks, all of the theorems from algebraic geometry which we need, except for the theory of the Shafarevich fibration, which we will review in Section~\ref{section:Shafarevich}.

Section~\ref{section:rational.curves} summarizes our previous theorem which, roughly, says that we can assume that~$M$ contains no rational curves.
In other words, we can reduce the study of holomorphic Cartan geometries on
compact complex manifolds $M$ to the study of such geometries on those compact complex manifolds $M$ which contain no rational curves.
(In the statement of Theorem~\ref{theorem:main}, we unpack all of our results to see what they imply even in the presence of rational
curves, but the reader should feel free to suppose that there are no rational curves henceforth.)

Purely for the reader's curiosity, in Section~\ref{sec:Cartan.on.proj}, we give a survey of the main theorems known about Cartan
geometries on smooth projective varieties, explaining how various fundamental geometric conditions on a smooth projective variety
constrain their Cartan geometries.

Section~\ref{section:the.anticanonical.splitting} provides detailed definitions, proofs and examples of some invariant fibrations of complex homogeneous manifolds.
The reader recalls that a Cartan geometry is ``infinitesimally modelled'' on a homogeneous space.
So these fibrations determine holomorphic foliations on complex manifolds with holomorphic Cartan geometries.
We finish this section by describing the geometry of these foliations.

In Section~\ref{section:splitting}, we will see that, roughly speaking, the biholomorphisms of any compact complex manifold $M$ without rational curves arise from splitting $M$
into a product $M = M_0\times A$ of a~complex torus $A$, on which translations are biholomorphisms, and a complex manifold $M_0$ with
finite automorphism group; see Lemma~\ref{lemma:symmetry.splitting}.

Section~\ref{section:inf.alg.models} produces a class of homogeneous models, larger than merely complex algebraic homogeneous models, to which are results will still apply.
Our results use a great deal of algebraic geometry, so they do not apply to arbitrary homogeneous models $X = G/H$, but they hold for complex algebraic homogeneous models.

Section~\ref{section:stability} defines various notions of stability of holomorphic principal bundles.
We use stability in Section~\ref{section:bundles.on.tori} to relate existence of a holomorphic connection on a bundle over a torus and existence of a reduction of the Levi quotient bundle with a flat holomorphic connection.
We use this in Section~\ref{section:Cartan.torus.fib} to show that a holomorphic Cartan geometry on a torus fibration has Levi quotient pulled back from the base, and is flat on the fibers.
This seems to be of independent interest; it is not used further, but gives hope that some picture like that of a~Shafarevich fibration might occur more generally.
It is also of interest because the known examples of holomorphic Cartan geometries without rational curves, other than the locally Hermitian symmetric varieties, are torus fibrations.

In Section~\ref{section:Large.fundamental.groups}, we show that any rational map of a flat Cartan geometry whose typical fiber has solvable holonomy has trivial pullback canonical bundle on its typical fiber.
We use this in Section~\ref{section:Shafarevich} to show that the Shafarevich fibration is actually defined at every point of $M$, and maps to a smooth base with ample canonical bundle.
This is the most technically involved section of the paper.
The subsequent sections easily put together a picture of the varieties which develop to the model.

Section~\ref{section:Shafarevich} will review the well known theory of Shafarevich fibrations from algebraic geometry~\cite{Kollar:1993,Zuo:1999}.
This theory takes as input the group structure of the fundamental group of a~complex manifold and gives as output a holomorphic fibration.

\section{Review of the theory of Cartan geometries}\label{section:review}

\subsection{Definition}

In this section, we provide a complete definition of Cartan geometry and of parabolic geometry.
For introductions to Cartan geometries, see \cite[Chapter~1]{Cap/Slovak:2009}, \cite{mckay2023introduction,Sharpe:2002}.
Suppose that $G$ is a~complex Lie group acting holomorphically and transitively on a complex manifold $X$;
the pair~$(X, G)$ is called a \emph{complex homogeneous space}.
Let $H \subset G$ be the stabilizer of a point~$x_0 \in X$.
A holomorphic $\pr{X, G}$-geometry, also known as a holomorphic \emph{Cartan geometry} modelled on~\pr{X, G}, on a manifold $M$ is a
choice of holomorphic principal $G$-bundle $E_G \to M$ with a holomorphic connection
$\omega$ and holomorphic $H$-subbundle $E \subset E_G$ so that the tangent spaces of $E$ are complementary to the horizontal
spaces of the connection $\omega$ \cite{Sharpe:2002}.
The \emph{Cartan connection} is the restriction to $E$ of the connection $1$-form $\omega$ on $E_G$.
A Cartan geometry is \emph{effective} if the action of $G$ on $X$ is so.

The principal $H$-bundle $G \to X$ defined by $g \longmapsto gx_0$ is a Cartan geometry, called the \emph{model
geometry}, with the left invariant Maurer--Cartan $1$-form $g^{-1} {\rm d}g$ on $G$ as Cartan connection.

If $X' \subset X$ is a connected component and $G' \subset G$ is the subgroup preserving $X'$, then any \pr{X, G}-geometry
is precisely an \pr{X', G'}-geometry, so we can assume, without loss of generality, that the model $X$ of any Cartan geometry is connected.
A \emph{flag variety} is a complex homogeneous space $(X, G)$ for which $X$ is a connected and simply connected projective variety
and $G$ is a~complex semisimple Lie group; they are precisely the rational homogeneous projective varieties.
A \emph{parabolic geometry} is a Cartan geometry whose model is a flag variety \cite{Cap/Slovak:2009}.

\subsection{Dropping}
In this section, we will see that each Cartan geometry gives rise, by a trivial construction, to another one, on a higher-dimensional manifold, which we call the \emph{lift} \cite[Section 9]{mckay2023introduction}.
If a Cartan geometry arises from the lift of another, we say it \emph{drops} to that other Cartan geometry.

Suppose that $H \subset H' \subset G$ are closed complex subgroups of a complex Lie group $G$, with associated homogeneous
spaces $X := G/H$ and $X' := G/H'$, and that $(E, \omega) \to M'$ is an \pr{X', G}-geometry.
If we define $M := E/H$, then $E \to M$ is an \pr{X, G}-geometry with the same Cartan connection $\omega$;
we call it the \emph{lift} of $(E, \omega) \to M'$ \cite[Section~1.5.13]{Cap/Slovak:2009}.
(The lift is traditionally called the \emph{correspondence space} of $(E, \omega) \to M'$ \cite[Section~1.5.13]{Cap/Slovak:2009}, but this is an awkward fit to our applications of the idea, as we will see).
Note that $M \to M'$ is a holomorphic fiber bundle
with fiber $H'/H$, the same fiber as the $G$-equivariant holomorphic map $X \to X'$.
Conversely, a given $\pr{X, G}$-geometry \emph{drops} to a given $\pr{X',G}$-geometry if it is isomorphic to the $\pr{X, G}$-lift
of that $\pr{X', G}$-geometry.
(The notion of \emph{dropping} is fundamental to all of our work in this paper, more so than lifting. It would be awkward to use the expression \emph{is isomorphic to the correspondence space of} instead of \emph{drops to}.
Therefore, we adopt the terminology of \emph{lifting} and \emph{dropping} henceforth.)

\begin{Example}
If $X$ is a connected homogeneous $G$-space and $*$ is a point with trivial $G$-action, then an $(X, G)$-geometry on a
connected manifold drops to a $(*, G)$-geometry just when the geometry is isomorphic to its model.
\end{Example}

More generally, if a geometry on some complex manifold $M$ drops to some complex manifold~$M'$, then we can recover the complex manifold $M$ and the original geometry on $M$ directly from the geometry on $M'$; see \cite{Biswas.McKay:2016} for details and \cite{McKay2011b} for examples.
A \emph{minimal geometry} is a~holomorphic Cartan geometry which does not drop to a holomorphic Cartan geometry on any lower-dimensional manifold.

\subsection{From modules to vector bundles}
Any Cartan geometry is modelled on a homogeneous space.
In this section, we explain that various vector bundles naturally arise in any Cartan geometry, and that these vector bundles are modelled on homogeneous vector
bundles \cite[Section 10]{mckay2023introduction} on the homogeneous space.
Take a~Cartan geometry $E \to M$ modelled on a complex homogeneous space $(X, G)$ with~${X = G/H}$.
The principal $H$-bundle $E \to M$ and the Cartan connection determine the Cartan geometry,
as $E_G := \prodquot{E}{G}{H}$ and the Cartan connection extends uniquely to $E_G$ to become a holomorphic connection $1$-form.
To any holomorphic $H$-module $V$ we associate the holomorphic vector bundle~$\vb{V} := \prodquot{E}{V}{H} \to M$
whose global sections are the holomorphic $H$-equivariant maps~${E \to V}$.
We use the same symbol $\vb{V}$ for the associated vector bundle on the model $X$ as well.

\begin{Example} If $V := \LieG/\LieH$ equipped with the adjoint action of $H$, then $\vb{V} = \vbTM = TM$
\cite[Proposition 1]{mckay2023introduction}, \cite[Theorem 3.15]{Sharpe:1997}.
\end{Example}

\subsection{Development}\label{subsection:Development}
Suppose that $M$ is a complex manifold with a holomorphic Cartan geometry.
In this section, we consider all holomorphic maps to $M$, from any complex manifold, on which the pullback curvature of the Cartan geometry vanishes.
We call these maps \emph{pancakes}.
Such a map, in some sense, locally resembles a map to the model of the Cartan geometry, its \emph{development}.
All material in this section is explained in more detail and with complete proofs
in \cite[Section 7]{mckay2023introduction}.

Take a complex homogeneous space $(X, G)$ with marked point $x_0 \in X$.
Take a holomorphic~$(X, G)$-Cartan geometry $H \longrightarrow
E \to M$ with Cartan connection $\omega$.
A \emph{pancake} of the Cartan geometry is a holomorphic map $f \colon Z \longrightarrow
 M$ from a connected complex space so that the holomorphic connection $\omega$ on $E_G := \amal{E}{H}{G}$ pulls back to a flat
connection on $f^* E_G \to Z$. Pick a point $z_0 \in Z$.
\begin{Example}
Every holomorphic map $Z\to X$ of any complex manifold $Z$ to a homogeneous space $(X,G)$ is a pancake for the model Cartan geometry on $X$.
So maps $Z\to M$ which are not pancakes, roughly speaking, ``look different'' locally than maps $Z\to X$, hence the pancake condition.
\end{Example}
There is a unique unramified Galois covering
\[
p_Z \colon \ \big(\widehat{Z}, \widehat{z}_0\big) \to (Z, z_0)
\]
so that the flat connection has trivial holonomy precisely on covering spaces of $\widehat{Z}$.
Each choice of point \smash{$e_0 \in \pr{f \circ p_Z} ^* E_{\widehat{z}_0} = E_{z_0}$} gives an $H$-equivariant morphism $\widehat{f} \colon
 \pr{f \circ p_Z}^*E \to G$ of complex spaces, uniquely determined by
$
\widehat{f}\of{e_0} = 1
$
and $\widehat{f}^* g^{-1} {\rm d}g = \omega$ \cite[Section 2.3]{Biswas.McKay:2016}. In fact~$e_0$ produces an
isomorphism of $\pr{f \circ p_Z}^*E_G \to \widehat{Z}\times G$ of principal $G$-bundles, and
$\widehat{f}$ is the composition of maps
\[
\pr{f \circ p_Z}^*E \hookrightarrow \pr{f \circ p_Z}^*E_G = \widehat{Z}\times G \longrightarrow G.
\]
By $H$-equivariance, $\widehat{f}$ descends to a morphism
\begin{equation}\label{ed}
\widehat{f} \colon \ \big(\widehat{Z}, \widehat{z}_0\big)
\to (X, x_0) = (G/H, eH)
\end{equation}
of complex spaces; it is known as the \emph{developing map}.

Let
\[
\pi := \text{Gal}(p_Z) = \pi_1(Z)/\pi_1\bigl(\widehat Z\bigr)
\]
be the Galois group. The flat connection has holonomy morphism
$
h \colon \pi \to G
$
uniquely determined by $\widehat{f}\circ\gamma = h(\gamma)\widehat{f}$ for every $\gamma \in \pi$, where $\widehat{f}$ is the map in \eqref{ed}.
We have an exact sequence of group morphisms
\[
\begin{tikzcd}
1\ar[r]& \pi_1\bigl(\widehat{Z}, \widehat{z}_0\bigr) \ar[r]&\fundamentalGroup{Z, z_0}\arrow{r}{h}&G.
\end{tikzcd}
\]
The developing pair $\widehat{f}, h$ is \emph{associated} to the pancake $f \colon Z \to M$.
The pancake \emph{develops to the model} if $Z = \widehat{Z}$, i.e., if $\pi = 1$.

The group $\pi$ acts on $\widehat{f}^* G$ on the left by the holonomy morphism,
where $G$ is considered as a principal $H$-bundle on $X$ using the quotient map
$G \longrightarrow G/H$, and $\widehat{f}$ is the map in \eqref{ed}; quotienting: $f^*E \cong \big(\widehat{f}^*E\big)/\pi$, and $g^{-1} {\rm d}g$ descends to define a connection on $f^*E_G$, which actually coincides with $f^*\omega$. The developing pair
\[
\big(\widehat{f}, h\big), \qquad
\widehat{f} \colon \big(\widehat{Z}, \widehat{z}_0\big) \to (X, x_0),\qquad
h \colon\ \pi \to G
\]
together with choice of points $Z \ni z_0 \longmapsto m_0 \in M$ and the Cartan connection on $M$ determines uniquely
the original map $Z \to M$.

The developing pair of a pancake is uniquely determined by choice of point $e_0 \in \! \pr{f \circ p_Z} ^* E_{\widehat{z}_0}
 \!= E_{m_0}$. If we replace $e_0$ by another point $e_0 h_0$ for some $h_0 \in H$, then
\smash{$\big(\widehat{f}, h\big)$} gets replaced by \smash{$\big(h_0^{-1} \widehat{f}, h_0^{-1} h h_0\big)$}.
If we replace $z_0$ by another point of $Z$, similarly \smash{$\big(\widehat{f}, h\big)$} gets replaced by~\smash{$\big(g^{-1} \widehat{f}, g^{-1} h g\big)$} for some $g \in G$.
Every pancake has a unique developing pair, up to this right $G$-action \cite[Section~2.3]{Biswas.McKay:2016}.

The following pullback bundles have canonical isomorphisms of bundles on $\hat{Z}$ \cite[Section~2.3]{Biswas.McKay:2016}:
\begin{gather*}
\begin{split}
&f^* E \cong \big(\widehat{f}^* G\big)/\pi \text{ as holomorphic principal $H$-bundles}, \\
&f^* TM \cong \big(\widehat{f}^* TX\big)/\pi \text{ as holomorphic vector bundles}, \\
&f^* \pr{\amal{E}{H}{\LieG}} \cong \widehat{f}^* \pr{\amal{G}{H}{\LieG}}/\pi \text{ as holomorphic vector bundles with connection}.
\end{split}
\end{gather*}
If $Z$ develops to the model, these obviously become isomorphisms of bundles on $Z$:
\begin{gather*}
\begin{split}
&f^* E \cong \widehat{f}^* G \text{ as holomorphic principal $H$-bundles}, \\
&f^* TM \cong \widehat{f}^* TX \text{ as holomorphic vector bundles}, \\
&f^* \pr{\amal{E}{H}{\LieG}} \cong \widehat{f}^* \pr{\amal{G}{H}{\LieG}} \text{ as holomorphic vector bundles with connection}.
\end{split}
\end{gather*}
We will make crucial use of these isomorphisms below.

A submanifold $Z\subseteq M$ \emph{develops to the model} if $Z$ is a pancake and the inclusion map~${Z\!\to\! M}$ develops to the model.

For more on the story of development of curves to the model from flat Cartan geometries, see \cite[p.~3]{Goldman:2010}. For development to the model of curves in non-flat Cartan geometries, see
\cite[p.~108, 1.5.17]{Cap/Slovak:2009}, \cite[p.~54, 41X]{Cartan1935}, \cite{Ehresmann:1951}, \cite[p.~368]{Sharpe:1997}. For developments of higher-dimensional submanifolds, still only to the model, but only
in flat Cartan geometries (which suffices for our applications in this paper), see
\cite[Section~4]{KobayashiOchiai:1980}. The concept of development of curves was first used to develop
curves in one geometry to curves in another, with neither being assumed to be a~homogeneous model
\cite{Rinow1932}, but this was only carried out in Riemannian geometry to compare any two Riemannian
geometries on surfaces. For development of curves between arbitrary Cartan geometries with the same model,
the only reference is \cite[Section~17]{mckay2023introduction}. For development of higher-dimensional
submanifolds, in Cartan geometries which might not be flat, still only developing to the model, the only
references are \cite[Section~2.3]{Biswas.McKay:2016}, \cite[p.~43]{mckay2023introduction}, the second of
which gives complete proofs of the results we use here.

\section{Review of algebraic geometry}\label{section:algebraic.geometry}
In this section, we recall some theorems of complex algebraic geometry.
We were not able to find all of this information in a single source, so we need to give a summary of the results that we will use.
\subsection{Line bundles}
Suppose that $M$ is a compact irreducible reduced complex space.
The \emph{Iitaka dimension} of a~holomorphic line bundle $L$ on $M$ is the maximal dimension of the images of the rational maps
\[
\rationalMap{M }{ \Proj{}\cohomology{0}{M, kL}^*}
\]
taking a point $m \in M$ to the hyperplane in $\cohomology{0}{M, kL}$ consisting of sections vanishing at $m$, for $k = 1,
2, \dots$; see \cite[p.~50]{Ueno:1975}, \cite[Definition~2.1.3]{Lazarsfeld:2004}.
If $0 = \cohomology{0}{M, kL}$ for all $k > 0$, the Iitaka dimension is defined to be $-\infty$ \cite[p.~50]{Ueno:1975}, \cite[Definition~2.1.3]{Lazarsfeld:2004}.
The line bundle $L \to M$ is \emph{big} if the Iitaka dimension of $L$ is the dimension of $M$ \cite[Definition~2.2.1]{Lazarsfeld:2004}.
A nef line bundle~${L \to M}$ is big just when $c_1(L)^n > 0$, where $n := \dimC{M}$; see \cite[Theorem~2.2.16]{Lazarsfeld:2004}.
The \emph{numerical dimension} of $L \to M$ is the largest integer $k$ for which
$0 \ne c_1(L)^k \in \cohomology{2k}{M, \R{}}$, so a nef line bundle is big precisely when its numerical dimension equals the dimension of the complex manifold $M$ \cite[Remark~2.3.17]{Lazarsfeld:2004}.
The \emph{Kodaira dimension} $\kappa_M$ of $M$ is the Iitaka dimension of the canonical bundle of $M$ see \cite[p.~65]{Ueno:1975}.
The Kodaira dimension is at most the numerical dimension; a holomorphic line bundle is \emph{abundant} if its Kodaira and numerical dimensions are equal \cite[Remark~2.3.17]{Lazarsfeld:2004}.
The space $M$ is of \emph{general type} if its canonical bundle is big \cite[Example 2.2.2]{Lazarsfeld:2004}.
Any compact complex manifold with $c_1 < 0$ has general type, as does any modification of it.
\subsection{The Iitaka fibration}
An \emph{Iitaka fibration} of $M$ is a dominant rational map $\rationalMap{M}{\Ii{M}}$, whose generic fiber is smooth and irreducible, so that subvarieties of $M$ on which $\cb{M}$ has zero Iitaka dimension lie in the fibers and, on the very general fiber, $\cb{M}$ has zero Iitaka dimension \cite[Theorem~2.1.33]{Lazarsfeld:2004}.
Any projective variety of nonnegative Kodaira dimension has a unique Iitaka fibration up to birational isomorphism
\cite[Theorem~10.3]{Iitaka:1982} or \cite[Theorem~5.10]{Ueno:1975},
and its fibers are connected \cite[p.~124]{Lazarsfeld:2004}.
If some power of the canonical bundle is spanned by global sections, then there is a unique holomorphic Iitaka fibration, and its codomain is the image of the $\cb[p]{M}$-maps
\[
M \to \Proj{}\cohomology{0}{M,\cb[p]{M}}^*
\]
for all but finitely many integers $p > 0$ \cite[p.~133]{Lazarsfeld:2004}.
\subsection{Moishezon manifolds}
In this section we explain that our theorems actually hold for a larger class of complex manifolds.
All varieties in this paper are assumed complex.
A \emph{Moishezon manifold} is a connected compact complex manifold bimeromorphic to a smooth projective variety \cite[p.~26]{Ueno:1975}.
If a Moishezon manifold $M$ admits a holomorphic Cartan geometry then $M$ is a smooth projective variety; \cite[Corollary~2]{Biswas.McKay:2016}.
If the holomorphic Cartan geometry on $M$ drops to a complex manifold $M'$ then $M'$ is also a smooth complex projective variety; \cite[Corollary~2]{Biswas.McKay:2016}.
Hence the classification of holomorphic Cartan geometries on Moishezon manifolds follows immediately from the classification on smooth projective varieties, which itself follows immediately from the classification of minimal geometries on smooth projective varieties.
Consequently, Theorem~\ref{theorem:main} holds true with \emph{smooth projective variety} replaced by \emph{Moishezon manifold}.
On the other hand, there are compact K\"ahler manifolds which are not projective and which admit holomorphic flat Cartan geometries.
For example, complex tori admit flat holomorphic projective connections.

Throughout this paper, we try to prove results in the greater generality of compact K\"ahler manifolds, but we are not always able to, since many of our results use theorems of algebraic geometry which are not understood for compact K\"ahler manifolds.
We expect the results of this paper hold with little modification for compact K\"ahler manifolds.

\subsection{Stein manifolds}
A \emph{Stein manifold} is a complex manifold for which all cohomology in all positive degrees of all coherent sheaves vanishes \cite[p.~100]{Grauert/Remmert:2004}.
A complex manifold is Stein just when it admits a~holomorphic proper embedding into complex Euclidean space \cite[Theorem~5.3.9]{Hormander:1990}.
We will need to make use the existence of such an embedding for some complex Lie groups below, so we will need to know which complex Lie groups are Stein.

\subsection{Rational curves}\label{subsection:rational.curves}
A \emph{rational curve} in a complex manifold $M$ is a nonconstant holomorphic map $\Proj{1}\longrightarrow M$ from the Riemann sphere.
A complex manifold is \emph{uniruled} if a rational curve passes through its generic point \cite[p.~181]{Kollar:1996}.
A line bundle $L \to M$ on a projective variety is \emph{nef} if $\int_C c_1(L) \ge 0$ for all complete complex algebraic curves $C$ in $M$ \cite[Definition~1.4.1]{Lazarsfeld:2004}.
A projective variety is \emph{minimal} if it has nef canonical bundle \cite[p. 2]{Cascini:2013}.
By Mori's cone theorem \cite{Cascini:2013,Mori:1982}, every smooth projective variety with no rational curves is minimal.
The \emph{abundance conjecture} claims that, on any minimal projective variety, some power of the canonical bundle is spanned by its global sections \cite{Siu:2009}.

\subsection{Rational connectivity}
A rational curve in a complex manifold is \emph{tame} if it lies in a compact irreducible component of the Douady space of deformations \cite{Douady:1966}.
In any compact complex manifold dominated by a~compact K\"ahler manifold, every rational curve is tame \cite{MR715653,MR973802}; in particular, every rational curve in any smooth projective variety is tame.
A compact complex manifold is \emph{tamely rationally connected} if two general points of it lie on a connected finite union of tame rational curves \cite{Kollar:1996}.
\begin{Example}
Any connected compact complex manifold with $c_1 > 0$ is tamely rationally connected \cite{Campana:1992}.
\end{Example}
A rational curve $f \colon \mathbb{P}^1\longrightarrow M$ in a complex manifold $M$ is \emph{ample} if $f^*TM$ is a sum of line bundles of positive degree.
A smooth projective variety is tamely rationally connected just when it contains an ample rational curve \cite[Definition 1.1, Theorems 1.9,~3.7 and~3.10]{Kollar:1996}.
A complex manifold is \emph{tame} if its every rational curve lies in a compact irreducible component of its Douady deformation space \cite[p.~2]{Biswas.McKay:2016}.

Every smooth projective variety is tame, as its Douady space is its Hilbert variety \cite{Douady:1966}.
\subsection{Affine groups}\label{subsection:affine.groups}
In this section, we invent a new class of complex Lie groups which behave enough like linear algebraic groups that our theorems will easily extend to them.
A complex Lie group $G$ is \emph{affine} if some finite index subgroup of $G$ has a morphism of complex Lie groups to a complex linear algebraic group, which is an isomorphism on Lie algebras.
\begin{Example}
Any complex linear algebraic group is affine, and hence any complex semisimple Lie group is affine.
\end{Example}
\begin{Example}
Every covering group of any complex linear algebraic group is affine.
For example, if a complex linear algebraic group is connected, it admits a universal covering group.
Often even a disconnected complex linear algebraic group has a universal covering group \cite[Theorem~18.2.1, Example~18.2.2]{Hilgert.Neeb:2012}.
\end{Example}
\begin{Example}
Products of affine are affine.
\end{Example}
\begin{Example}
Any complex Lie group which admits a holomorphic representation, injective on its Lie algebra, as unipotent complex
linear transformations, is affine \cite[Corollary 14.38]{Milne:2017}.
\end{Example}
\begin{Example}
Any complex Lie group with finite fundamental group and finitely many components is affine
\cite[Corollary 16.3.9]{Hilgert.Neeb:2012}.
\end{Example}

Suppose that $G$ is affine, say with finite index subgroup $G_0 \subset G$ having morphism $G_0 \longrightarrow
\overline{G}_0$ to a complex linear algebraic group, inducing an isomorphism on Lie algebras.
If $G \to P \to M$ is a holomorphic principal $G$-bundle, then on the finite unramified covering
space \smash{$\widehat{M} := P/G_0$}, we have a holomorphic principal $G_0$-bundle $G_0 \to P
\to \widehat{M}$, and an associated principal $\overline{G}_0$-bundle $\overline{P} := \amal{P}{G_0}{\overline{G}_0}$.
Every holomorphic connection on $P \to M$ pulls back to a holomorphic connection on $P
 \longrightarrow \widehat{M}$, and conversely by averaging over the deck transformations,
every holomorphic connection on $P \to \widehat{M}$ induces a holomorphic connection on $P \to M$.
We can replace $G_0$ by a finite index subgroup, hence assume that $\overline{G}_0$ is connected.
So, up to finite covering, the existence of a holomorphic connection, or of a holomorphic flat connection, is the same for the
original bundle as for the associated bundle with connected complex linear algebraic structure group.

\section{Rational curves and Cartan geometries}\label{section:rational.curves}
We recall one of our earlier results, which is essential for the results below:
\begin{Theorem}[{\cite[Theorem 2]{Biswas.McKay:2016}}]\label{theorem:drop}
On a smooth projective variety $M$ bearing a holomorphic Cartan geometry, the following are equivalent:
\begin{enumerate}\itemsep=0pt
\item[$(1)$]
Some holomorphic Cartan geometry on $M$ is minimal.
\item[$(2)$]
Every holomorphic Cartan geometry on $M$ is minimal.
\item[$(3)$]
$M$ contains no rational curves.
\item[$(4)$]
$M$ is not uniruled.
\item[$(5)$]
$M$ is not the total space of a holomorphic fiber bundle, with positive-dimensional flag manifold fibers, over a
minimal smooth projective variety.
\item[$(6)$]
$M$ is minimal, i.e., has nef canonical bundle.
\end{enumerate}
\end{Theorem}
Therefore, the classification of holomorphic Cartan geometries on smooth projective varieties follows from the classification of minimal geometries on minimal smooth projective varieties containing no rational curves.

\section{Review of Cartan geometries on projective varieties}\label{sec:Cartan.on.proj}
We set the stage by summarizing some theorems about Cartan geometries on smooth projective varieties.

\subsection{Positive Ricci}
In this section, we summarize the state of the art about holomorphic Cartan geometries on compact complex manifolds with positive Ricci curvature.
\begin{Theorem}[{\cite[Corollary 3]{Biswas.McKay:2016}}]
Any tamely rationally connected compact complex manifold $M$ admits a holomorphic Cartan geometry, say with model $(X, G)$, if and only if $M$ is biholomorphic to $X$, in which case the geometry is isomorphic to the model geometry and $(X, G)$ is a flag variety.
\end{Theorem}

\subsection{Ricci flat}
In this section, we summarize the state of the art about holomorphic Cartan geometries on compact complex manifolds with zero Ricci curvature.
\begin{Theorem}[{\cite{Biswas.McKay:2010}}]\label{thm:cpt.Kaehler.c.1.0}
A compact connected K\"ahler manifold $M$ with $c_1(M)=0$ admits a holomorphic Cartan geometry if and only if it has an unramified covering by a complex torus.
\end{Theorem}

\begin{Theorem}[{\cite{Biswas.Dumitrescu:2017}}]\label{theorem:BD.torus}
If the group $G$ of a complex homogeneous space $(X, G)$ is affine, then every holomorphic $(X, G)$-geometry on any complex
torus is translation invariant.
\end{Theorem}

\subsection{Semipositive Ricci}
In this section, we summarize the state of the art about holomorphic Cartan geometries on compact complex manifolds with semipositive Ricci curvature.
\begin{Theorem}[{\cite[Theorem 2]{McKay:2016}}]\label{theorem:semipos}
Suppose that $M$ is a connected compact K\"ahler manifold with~${c_1 \ge 0}$.
After perhaps replacing $M$ by a finite unramified covering space, every holomorphic Cartan geometry on $M$ is the lift of a Cartan geometry on a complex torus, and in particular~$M$ is a holomorphic fiber bundle, with flag variety fibers, over a complex torus.
\end{Theorem}

\section{The anticanonical splitting}\label{section:the.anticanonical.splitting}
\subsection{The anticanonical fiber bundle}
In this section, we define the anticanonical fibration of a complex homogeneous space $X$, which is a fibration invariant under automorphisms; the anticanonical fibration was previously considered by~\cite{Huckleberry:1986}, but the results below are new.
We prove that every subvariety of a complex homogeneous space on which the ambient canonical bundle is trivial lies in a fiber of the anticanonical fibration.
We will subsequently apply this to subvarieties which are developed to $X$ from some Cartan geometry.

For any normal complex space $X$ and holomorphic line bundle $L \to X$, let
\[
V := \cohomology{0}{X, L},
\]
which might be an infinite-dimensional complex vector space.
Let $X' \subset X$ be the points of $X$ where every holomorphic section of $L$ vanishes.
If $\dimC{ V} > 0$, define the $L$-map
\[
X\setminus X' \longrightarrow \mathbb{P}V^*, \qquad x \longmapsto \{s \in V\mid s(x)=0\}.
\]
If $X' \subset X$ contains no component of $X$, denote the image of the $L$-map by $\Iitaka[L]{X}$ and the $L$-map
by~$\rationalMap[\Iitaka[L]{}]{X}{\Iitaka[L]{X}}$.
For any two points $x_0, x_1 \in X\setminus X'$ in the same fiber, we can pick a~vector~${v_0 \in L_{x_0}}$, and a global
section $s$ with $s(x_0) = v_0$, and hence define $v_1 := s(x_1)$, a natural linear isomorphism $L_{x_0}
 \cong L_{x_1}$, i.e., $L$ is trivial along the fibers of $\rationalMap[\Iitaka[L]{}]{X}{\Iitaka[L]{X}}$.
For every complex homogeneous space \pr{X, G} and $G$-invariant line bundle $L \to X$, the map $\rationalMap[\Iitaka[L]{}]{X}{\Iitaka[L]{X}}$ is $G$-equivariant.
The map is everywhere defined, by $G$-equivariance, as long as $L\to X$ has at least one holomorphic section not everywhere vanishing.
The fibers of $X \to \Iitaka[L]{X}$ are the equivalence classes: $x \sim x'$ just when every holomorphic section of $L$ which vanishes at $x$ also vanishes at $x'$ and vice versa.
The sections vanishing at $x$ have common vanishing locus a subvariety $Z_x$, and each fiber $F = X_{\overline{x}_0}$ is
\[
F = \bigcap_{x \in F} Z_x,
\]
a closed subvariety acted on by the normalizer of $H \subset G$.
The fiber $F$ is a closed complex analytic subvariety, so the stabilizer of $F$ is a closed complex subgroup $\Iitaka[L]{H}
 \subset G$, and therefore
 \[\Iitaka[L]{X} = G/\Iitaka[L]{H}\]
 is a complex homogeneous space and $X \to \Iitaka[L]{X}$ is the \emph{$L$-fiber bundle}.
The line bundle $L$ is pulled back by $\rationalMap[\Iitaka[L]{}]{X}{\Iitaka[L]{X}}$ from a line bundle on $\Iitaka[L]{X}$.
In particular, we will consider $L = \acb{X}$, i.e., $X \longrightarrow \acf{X}$, the \emph{anticanonical fiber bundle}.
Note that $X$ has global nonzero sections of its anticanonical bundle, as we can wedge together vector fields from the $G$-action.
In particular, the anticanonical fiber bundle is a well defined regular morphism.
\begin{Example}
Some easy examples, which essentially exhaust the complex homogeneous spaces in low dimension \cite{McKay:2015}:
\begin{enumerate}\itemsep=0pt
\item[(1)]
Any flag variety $X = G/P$ has positive anticanonical bundle, so $\acf{X} = X$.
\item[(2)]
If $X = \Proj{1} \times \C \times C$ where $C = \C/\Lambda$ is an elliptic curve, then $\acf{X}
 = \Proj{1} \times \C$ giving the fiber bundle $C \to X \to \acf{X}$.
\item[(3)]
If $(X, G)$ is a homogeneous Hopf manifold of dimension $n$ then $\acf{X} = X/\C^{\times} = \Proj{n-1}$.
To be precise,
\[X = (\C^{n}-\{0\})/(z\sim\lambda z)\]
 for some $\lambda \in \C^{\times}$ with $|\lambda| \ne 1$ (which
can be chosen arbitrarily to construct a homogeneous Hopf manifold), and $H^0(X, -K_X)$ consists of the tensors
\[
p(z)\partial_1\wedge\dots\wedge\partial_n
\]
for any polynomial $p(z)$ homogeneous of degree $n$.
Such a tensor, by homogeneity, vanishes on a set of complex lines through the origin in $\C^{n}$.
Hence the anticanonical fibers are the quotients of these complex lines, i.e., the fibers of the Hopf fibration $X \longrightarrow \acf{X}
 = \Proj{n-1}$, and are elliptic curves $\C^{\times}/\langle \lambda\rangle$.
\item[(4)]
The total space $X = \OO[n]$ of the line bundle $\OO[n] \to \Proj{1}$ for $n \ge 1$ has $\acf{X} = X$.
\item[(5)]
There is one nontrivial holomorphic fiber bundle $\C^{\times} \to X \to \C^{\times}$, and it admits infinitely many holomorphic faithful group actions of connected complex Lie groups \cite{McKay:2015}.
Its anticanonical fibration is $\C^{\times} \to X \to \C^{\times} = \acf{X}$.
\end{enumerate}
\end{Example}

\begin{Lemma}\label{lemma:all.points}
For any complex homogeneous space \pr{X, G}, any vector field on $X$ arising from an element of the Lie algebra of $G$ vanishes
at a point $x \in X$ just when it vanishes at all points of the anticanonical fiber through $x$.
\end{Lemma}

\begin{proof}
Denote by $\LieG$ the Lie algebra of $G$.
Pick a vector field $v \in \LieG$ and points $x, y \in X$ so that~${v(x) = 0}$ and $v(y) \ne 0$.
Let $v_1 := v$.
Since $G$ acts transitively on $X$, we can find vector fields $v_2, v_3, \dots, v_n$ so that $v_1, v_2,
 \dots, v_n$ are linearly independent at $y$.
By slight perturbation, we can arrange that $v_2, v_3, \dots, v_n$ are linearly independent at $x$.
Then $v_1 \wedge \dots \wedge v_n$ vanishes at~$x$ but not $y$.
So $x$ and $y$ lie in different fibers of the anticanonical fibration.
\end{proof}

Consequently, if we take a point $x \in X$ then its anticanonical fiber $F$ has vector fields
\[
\LieG \longrightarrow \LieG/\LieG^x \subset \cohomology{0}{F, TX}, \qquad
v \longmapsto v|_F
\]
giving a canonical trivialization of $ TX|_F$ along each anticanonical fiber $F$.
A vector field from~$\LieG$ vanishes at $x_0$ just when it belongs to $\LieH$ and just when it vanishes on the anticanonical
fiber~${F = \acf{H}x_0}$.
So $\LieH$ is precisely the subalgebra of $\LieG$ acting trivially on $\acf{H}x_0 = \acf{H}/H$.
Denote by $\acf\LieH$ the Lie algebra of $\acf{H}$; so $\LieH \subset \acf\LieH$ is an ideal.

The elements of $\acf{H}$ leaving every point of $F$ fixed form a normal subgroup of $\acf{H}$, containing $H^0$ and having Lie algebra $\LieH$.
So this normal subgroup has identity component $H^0$, and hence $H^0 \subseteq \acf{H}$ is a normal subgroup.
Thus
\[
H/H^0 \subset \acf{H}/H^0
\]
is a discrete group, giving the fiber $F = \acf{H}x_0$ the form
\[
F = \pr{\acf{H}/H^0}/\pr{H/H^0}
\]
of a quotient space of $\acf{H}/H^0$ by the action of a discrete group $H/H^0$.
In particular, the vector fields $\acf\LieH$ are a canonical $G$-invariant trivialization of the tangent bundle of each anticanonical fiber.

If $f \colon X \to \C$ is a holomorphic function with $f(x_0) \ne 0$ and $s$ is a section of the anticanonical bundle not vanishing at some point $x_0 \in X$, then $fs$ is another such section, so $f\ne 0$ on the fiber of $x_0$.

\begin{Lemma}
Take a complex homogeneous space $X = G/H$.
The identity component $\big(\acf{H}\big)^0 \subset \acf{H}$ lies in the identity component $N^0$ of the
normalizer $N := N_G H^0$ of the identity component $H^0 \subset H$.
\end{Lemma}

\begin{proof}
Pick a point $x \in X$ and vector fields $v_1, v_2, \dots, v_n \in \LieG$, linearly independent at the generic point
but with $v_1(x) = 0$ while $v_2, v_3, \dots, v_n$ are linearly independent at $x$.
Let
\[
s := v_1 \wedge v_2 \wedge \dots \wedge v_n.
\]
For any element $g \in \acf{H}$, since $s(x) = 0$, we must have $s(gx) = 0$.
So for any element $w \in \acf\LieH$, $\LieDer_w s$ vanishes at $x$:
\begin{align*}
\LieDer_w s
& =
\pr{\LieDer_w v_1} \wedge v_2 \wedge \dots \wedge v_n - v_1 \wedge \pr{\LieDer_w v_2} \wedge v_2 \wedge \dots \wedge v_n + \cdots,
\end{align*}
so that at $x$:
\[
0 = \pr{\LieDer_w s}(x) = \pr{\LieDer_w v_1}(x) \wedge v_2(x) \wedge \dots \wedge v_n(x).
\]
Since $v_2, \dots, v_n$ can be generically chosen, $\LieDer_w v_1$ vanishes at $x$.
Hence $\LieDer_w \LieH \subset \LieH$, i.e., $w$ is in the Lie algebra of $N_{\LieG} \LieH$.
So $\acf\LieH \subset N_{\LieG} \LieH$.
\end{proof}

\begin{Corollary}\label{corollary:trivial.conn}
Take a complex homogeneous space $X = G/H$.
In the holomorphic $H$-mod\-ule~$\LieG/\LieH$, the subspace $\acf\LieH/\LieH$ is acted on trivially by $H^0$, so the action
of $H$ on $\acf\LieH/\LieH$ descends to an action of the group of components $H/H^0$ on $\acf\LieH/\LieH$.
\end{Corollary}

\begin{proof}
Every element of $\acf\LieH$ lies in the normalizer $N_{\LieG} \LieH$, so for any $v \in \LieH$ and $w \in
\acf\LieH$, $[v, w] \in \LieH$.
Hence $\LieH$ acts trivially on $\acf\LieH/\LieH$ in the adjoint representation on $\LieG/\LieH$.
\end{proof}

\begin{Lemma}\label{lemma:in.the.fiber}
Take a complex homogeneous space $X = G/H$.
Then every holomorphic map $f \colon Z \to X $ from a connected normal compact complex space for
which $f^*\cb{X}$ is trivial lies in a fiber of the anticanonical fibration $X \to \acf{X}$.
\end{Lemma}

\begin{proof}
Any section of $\acb{X}$ vanishing at one point of $f(Z)$ pulls back to the zero section of~$f^*(\acb{X})$.
\end{proof}

\begin{Lemma}\label{lemma:Stein.G}
Every complex Lie group with countably many components whose identity component is simply connected is a Stein manifold.
\end{Lemma}

\begin{proof}
Clearly it is sufficient to check the identity component.
By Levi decomposition (see \cite[Theorem 16.3.7]{Hilgert.Neeb:2012}), the radical is simply connected.
The center belongs to the radical, and the radical contains no complex torus, so the center is reductive, and we apply
\cite[Theorem 1]{Matsushima/Morimoto:1960}.
\end{proof}

\subsection{Developing and anticanonical fibrations}

\begin{Lemma}\label{lemma:ab.var.in.fiber}
Take a complex homogeneous space $X = G/H$.
Take a finite holomorphic map $f \colon Z\to X$ from a connected projective variety for which $\cb{Z}$ and $f^*\cb{X}$ are trivial.
Then $f(Z)$ lies in a fiber of the anticanonical fibration $X\to \acf{X}$.
If $Z$ is smooth then it is an abelian variety, and the map $f$ is an immersion, equivariant for a morphism of complex Lie groups from a complex abelian Lie group acting transitively on the abelian variety $Z$, and the image of $f$ is an abelian subvariety of $X$.
\end{Lemma}

\begin{proof}
By Lemma~\ref{lemma:in.the.fiber}, $f$ maps to a fiber.
As above, the subgroup preserving the fiber acts transitively, with discrete stabilizer, so we can replace $X$ by that fiber, and then without loss of generality $X=G/\Gamma$ for some connected complex Lie group $G$ and discrete subgroup ${\Gamma \subset G}$.
After perhaps replacing by a finite unramified covering, any smooth projective variety $Z$ with trivial canonical bundle splits into a product $Z=A \times B$ where $A$ is an abelian variety and~$B$ is a simply connected Calabi--Yau manifold \cite{Bogomolov:1974}.
So the map $Z\to X$ lifts on each subvariety~$\{a_0\} \times B$, for any $a_0 \in A$, to a map $B
\to \widetilde{G}$ to the universal covering group of the identity component of $G$.
But that universal covering group is a Stein manifold by Lemma~\ref{lemma:Stein.G}, so embeds holomorphically in Euclidean space.

Since $B$ is a compact complex manifold, every holomorphic map of $B$ to Euclidean space is constant, by the maximum principle.
So the map $B\to\widetilde{G}$ is constant.
Hence $B$ is a point, so $Z = A$ is an abelian variety.
Trivialize the tangent bundle of $A$ with a basis of holomorphic vector fields $v_a$, and trivialize the tangent bundle of $X
 = G/\Gamma$ with a basis of holomorphic vector fields $w_i$ from the Lie algebra $\LieG$ of $G$.
The map $f \colon A \to X$ then gives relations $f_* v_a = c^i_a w_i$.
These $c^i_a \colon A \to \C$ are holomorphic functions, so constants.
In particular $f \colon A \to X$ is a~holomorphic immersion.
With a change of basis, we arrange that $v_a = w_a$.
\end{proof}

In the proof above, if $(X,G)$ is an algebraic homogeneous space, it is not clear to the authors whether $\tilde{G}$ is an affine Lie group.

\subsection{The ant fibration}\label{subsec:ant}

In this section, we define a new refinement of the anticanonical fibration, which we believe has independent interest.
Take the anticanonical fibration of the anticanonical fibration, i.e., fiberwise we look at two points of a fiber as equivalent if they have the same sections of the canonical bundle of that fiber vanishing at both.
Since the tangent bundle of each fiber is trivial, this is the same as saying that two points in the same fiber are equivalent if every holomorphic function on the fiber which vanishes at one vanishes at both.
This equivalence relation gets iterated, and once we finish, we have a homogeneous holomorphic fiber bundle with perhaps finer fibers, which we call the \emph{ant fibration}.
On each fiber of the ant fibration, all holomorphic functions are constant, and there is a unique holomorphic section of the anticanonical bundle, up to constant scaling.
Locally constant functions on any fiber are holomorphic, so the fibers of the ant fibration are connected.
Throughout our work below, the reader could make use of the anticanonical fibration instead of the ant fibration, with the advantage of being more familiar to experts in homogeneous complex manifolds, but the disadvantage of providing possibly looser restrictions on developing maps.

\begin{Example}
There is one nontrivial holomorphic fiber bundle $\C^{\times} \to X \to \C^{\times}$, and it admits
infinitely many holomorphic faithful group actions of connected complex Lie groups~\cite{McKay:2015}. Its
anticanonical fibration is $\C^{\times} \to X \to \C^{\times}$, as the holomorphic sections of the
canonical bundle are pulled back from the base of the fibration, which is easy to check from the explicit
description of $X$ in \cite{McKay:2015}. But the ant fibration is trivial $X \to X$, since the
anticanonical fiber~$\C^{\times}$ has its points separated by holomorphic functions.
\end{Example}

\begin{Lemma}\label{lemma:ant.Stein}
The universal covering space of each ant fiber is a connected and simply connected complex Lie group and hence a Stein manifold.
\end{Lemma}

\begin{proof}
Write the ant fibration as $X\to \overline{X} = G/\overline{H}$ and a fiber as \pr{X_0,\GZ} where $X_0 := \overline{H}/H$.
By Lemma~\ref{lemma:all.points}, adjoint $\overline{H}$-action preserves $\LieH$, so $H^0 \subseteq \overline{H}$ is a normal subgroup.
Hence $X_0$ is a~quotient of a complex Lie group by a discrete subgroup:
\[
X_0 = \overline{H}/H = \pr{\overline{H}/H^0}/\pr{H/H^0}.
\]
So the universal covering space of each fiber $X_0$ is a complex Lie group, and simply connected, so is a Stein manifold by Lemma~\ref{lemma:Stein.G}.
\end{proof}
Note that
\begin{align*}
&\cohomology{0}{X_0,\OO}=\C,\qquad
\cohomology{0}{X_0,T}=\LieG_0=\overline{\LieH}/\LieH,\qquad
\cohomology{0}{X_0,\acb{}}=\C.
\end{align*}

\begin{Lemma}\label{lemma:ab.var.in.fiber.2}
Take a complex homogeneous space $X = G/H$.
Take a finite holomorphic map~${f \colon Z \to X}$ from a connected projective variety for which $\cb{Z}$ and $f^*\cb{X}$ are trivial.
Then~$f(Z)$ lies in a fiber of the ant fibration.
Suppose that $Z$ is smooth.
The variety $Z$ is an abelian variety.
The map $f$ is an immersion, equivariant for a morphism of complex Lie groups from a complex abelian Lie group acting transitively on the abelian variety $Z$.
The image of $f$ is an abelian subvariety of $X$.
\end{Lemma}

\begin{proof}
By Lemma~\ref{lemma:ab.var.in.fiber}, $f$ maps to a fiber of the anticanonical fibration, so we can replace $X$ by that fiber, and $G$ by the automorphism group of that fiber, and repeat.
All holomorphic functions on $X$ pull back to constants on $Z$, so $f$ maps to a fiber of the ant fibration.
\end{proof}

\subsection{The ant foliation}
Since a Cartan geometry looks infinitesimally like its model, it bears a foliation which looks infinitesimally like the ant fibration.
In this section and the following sections we prove that the ant foliation splits the tangent bundle.

Take a holomorphic Cartan geometry $E \to M$ with model \pr{X,G} on a complex manifold~$M$.
Every holomorphic $G$-invariant vector subbundle $\vb{V} \subset TX$ has an associated
holomorphic vector subbundle $\vb{V} \subset TM$ by letting $V \subset T_{x_0} X$ be the
subspace of $\vb{V}$ at one point $x_0 \in X$, say with stabilizer $H \subset G$ and then
\[
\vb{V} := E \times^H V \subset E \times^H \pr{\LieG/\LieH}.
\]
If $\vb{F}$ is a $G$-invariant foliation of $X$, then its tangent bundle $\vb{V} := T\vb{F}$
becomes a subbundle~${\vb{V} \subset TM}$, but might not be tangent to a foliation of $M$.
The \emph{ant distribution} of $TM$ is the subbundle associated to the tangent bundle of the ant
foliation of $X$.

\begin{Lemma}
The ant distribution of any holomorphic Cartan geometry bears a flat holomorphic connection.
\end{Lemma}

\begin{proof}
Because $H^0$ acts trivially on $\acf\LieH/\LieH$, the associated vector bundle is
holomorphically trivial on $E/H^0$, a covering space.
Apply induction, taking anticanonical fibration of the anticanonical fibration.
\end{proof}

\begin{Lemma}
Take a smooth projective variety $M$ with a minimal geometry.
The ant distribution of the geometry is the tangent bundle of a foliation, the \emph{ant foliation} of the Cartan geometry, and admits a smooth holomorphic complementary foliation splitting $TM$.
\end{Lemma}

\begin{proof}
Since the ant distribution has a holomorphic connection, its Chern classes vanish.
Since~$M$ is not uniruled, every holomorphic distribution with trivial first Chern class is a smooth holomorphic foliation and admits a smooth holomorphic complementary foliation splitting $TM$ \cite[Theorem 5.2]{Loray/Pereira/Touzet:2011}.
\end{proof}

\section{Splitting}\label{section:splitting}

\begin{Lemma}\label{lemma:symmetry.splitting}
Take a smooth projective variety $M$ with a minimal geometry $E \to M$ modelled on a complex homogeneous space \pr{X,G}, $X = G/H$.
Let $V \subset \LieG/\LieH$ be the set of all vectors invariant under the $H$-action and let $\vb{V}:= E \times^H V \subset E \times^H \pr{\LieG/\LieH} = TM$.
Then $M$ splits~${M = M_0 \times A}$ where $A$ is an abelian variety and $M_0$ has no nonzero holomorphic vector fields: $0=\cohomology{0}{M_0,TM_0}$.
The distribution $\vb{V} \subset TM$ is the tangent bundle of a foliation on $M$.
Every leaf of that foliation lies inside an abelian variety $\{m_0\} \times A \subset M$.
\end{Lemma}

\begin{proof}
The distribution is $\vb{V} = E \times^H V \cong \pr{E/H} \times V$ trivial.
So each element $v \in V$ determines a nowhere vanishing holomorphic tangent vector field on $M$.
Since $M$ is minimal, i.e., contains no rational curves, $M$ is not uniruled (i.e., does not have a rational curve through the generic point).
A theorem of Liebmann \cite[Theorem 1.1]{Pereira/Touzet:2012} says that any smooth projective variety $M$ which is not uniruled splits as a product $M_0\times A$ where $A$ is an abelian variety, so that $M_0$ has zero as its only holomorphic vector field.

Since $A$ is a complex torus, we can write $A=W/\Lambda$ where $W$ is complex Euclidean space and~$\Lambda\subseteq W$ is a lattice.
Note that $TA$ is trivial: $TA=A\times W$.
Every tangent vector to $A$ extends uniquely to a nowhere vanishing holomorphic tangent vector field, a constant map to~$W$.

Since $TM=TM_0\times TA$, every holomorphic vector field $v$ on $M$ splits into a sum $v=(u,w)$ of a holomorphic vector field $u$ tangent to $M_0\times\{*\}$ and one $w$ tangent to $\{*\}\times A$.
The map
\[
p\in M\mapsto w(p)\in W
\]
maps a compact complex manifold holomorphically to Euclidean space, so is constant by the maximum principle.
Hence $v=u+w$ has $w$ constant.

For each fixed $a_0\in A$, the map
\[
p\in M_0\mapsto u(m_0,a_0)\in T_p M_0,
\]
is a holomorphic vector field, so vanishes.
Hence the holomorphic vector fields on $M$ are precisely the constant vector fields $(0,w)$ for a constant $w\in W$.
Since our distribution is spanned by global holomorphic tangent vector fields, these are of this form $(0,w)$, hence tangent to the tori~$\{*\}\times A$.
\end{proof}

\begin{Corollary}\label{cor:splitting}
Take a smooth projective variety $M$ with a minimal geometry $E \to M$ modelled on a complex
homogeneous space \pr{X,G}, $X = G/H$.
Take the ant fibration $X \to X' = G/H'$.
Suppose that $H$ has finitely many components.
Then, after perhaps replacing $M$ by a~finite unramified covering space, $M$ splits
$M = M_0 \times A$ where $A$ is an abelian variety and $M_0$ is a~smooth complex projective
variety with no nonzero holomorphic vector fields: $0 = \cohomology{0}{M_0,TM_0}$.
Every leaf of the ant foliation on $M$ lies inside an abelian variety~${\{m_0\} \times A \subset M}$.
\end{Corollary}

\begin{proof}
Replace $M$ by $E/H^0$, $H$ by $H^0$, with the same total space $E$ and the same Cartan connection, so that we can assume that $H$ is connected.
So $H$ acts trivially on $\LieH'/\LieH$.
\end{proof}

\section{Infinitesimally algebraic models}\label{section:inf.alg.models}
In this section, we define a class of complex homogeneous space which resemble complex linear algebraic homogeneous spaces.
We need a strong enough resemblance so that our theorems about Cartan geometries easily generalize to these models, even though the proofs use algebraic geometry methods.

Take a complex Lie group $G$ and some finite-dimensional holomorphic $G$-modules $V_i$.
A~\emph{reductive $($semisimple$)$ Levi quotient} of $G$ over $\{V_i\}$ is a quotient of complex Lie groups
\[
1 \to \urad{G} \to G \to \red{G} \to 1,
\]
and a complex Lie group morphism $\urad{G} \to \urad{\overline{G}}$ to a unipotent (solvable) complex
linear algebraic group, injective on Lie algebras, so that $\red{G}$ is a reductive (semisimple) complex
linear algebraic group and so that the action of $\urad{G}$ on each $V_i$ factors through a polynomial
action of $\urad{\overline{G}}$.
It follows that $\urad{G}$ preserves a filtration of $G$-modules on each $V_i$ by Lie's theorem
(see \cite[Theorem~3]{Serre:2001}), and acts trivially on the quotients $V_{i+1}/V_i$
\cite[Section 4.8]{Borel:1991}.
\begin{Example}

The group $G := \PSL{2,\C} \times \PSL{2,\Q{}}$ has no reductive or semisimple Levi quotient in the
usual sense of those terms (as defined in Vinberg \cite[p.~20]{Vinberg:1994}, for instance), nor does any finite index
subgroup \cite[Theorem 6.14]{Jacobson:1985}, but $G$ has a unique semisimple Levi quotient (also a~unique reductive Levi quotient) over $\LieG$:
\[
1 \to \PSL{2,\Q{}} \to G \to \PSL{2,\C} \to 1.
\]
\end{Example}
A complex homogeneous space $X=G/H$ is \emph{infinitesimally algebraic} if, after perhaps replacing $G$ and $H$ by finite index subgroups, $H$ has a reductive Levi quotient over $\LieG$, for which the action of $\red{H}$ on $\LieG/\LieH$ is faithful, and $G$ has a semisimple Levi quotient over $\LieG$.
\begin{Example}
Let $G$ be the set of all complex $3 \times 3$ matrices of the form
\[
\begin{pmatrix}
{\rm e}^a & 0 & b \\
0 & 1 & a \\
0 & 0 & 1
\end{pmatrix}
\]
and let $H \subset G$ be the subgroup of matrices of the form $a \in \pi {\rm i} \Z$, $b \in \Z$, i.e.,
\[
\begin{pmatrix}
(-1)^k & 0 & \ell \\
0 & 1 & \pi {\rm i} k \\
0 & 0 & 1
\end{pmatrix}
\]
for $k,\ell \in \Z$.
The quotient $X := G/H$ is the nontrivial holomorphic $\C^{\times}$-fiber bundle $\C^{\times} \to X
 \to \C^{\times}$ \cite[Section~11]{McKay:2015}.
We want to see that $(X, G)$ is infinitesimally algebraic, but not algebraic.
The adjoint action of an element
\[
g =
\begin{pmatrix}
{\rm e}^a & 0 & b \\
0 & 1 & a \\
0 & 0 & 1
\end{pmatrix}
\]
on
\[
A=
\begin{pmatrix}
\alpha & 0 & \beta \\
0 & 0 & \alpha \\
0 & 0 & 0
\end{pmatrix}
\in \LieG
\]
is
\[
\Ad_g A =
\begin{pmatrix}
\alpha & 0 & {\rm e}^a\beta-b\alpha \\
0 & 0 & \alpha \\
0 & 0 & 0
\end{pmatrix},
\]
so that in terms of vectors denoted
\[
\begin{pmatrix}
\alpha \\
\beta
\end{pmatrix},
\]
the adjoint representation factors through the solvable linear algebraic group of matrices
\[
\begin{pmatrix}
1 & 0 \\
-b & c
\end{pmatrix}
\]
with $c\ne 0$, by
\[
g \longmapsto
\begin{pmatrix}
1 & 0 \\
-b & {\rm e}^a
\end{pmatrix}.
\]
The semisimple quotient of $G$ is thus $G_{\rm ss} = 1$.
We let $\red{H} = \{\pm 1\}$, and map
\[
H \longrightarrow \red{H}, \qquad
\begin{pmatrix}
(-1)^k & 0 & \ell \\
0 & 1 & \pi {\rm i} k \\
0 & 0 & 1
\end{pmatrix}
 \longmapsto (-1)^k.
\]
The kernel $\urad{H}$ is the set of matrices of the form
\[
\begin{pmatrix}
1 & 0 & \ell \\
0 & 1 & 2\pi {\rm i} k \\
0 & 0 & 1
\end{pmatrix}
\]
for $k \in \Z$.
Hence the representation of $\urad{H}$ on $\Ad_G$ factors through
\[
\urad{H} \longrightarrow \urad{\overline{H}} := \left\{\begin{pmatrix}1&0\\-b&1\end{pmatrix}\right\}.
\]
The reader can check that the center of $G$ is the infinite discrete group
\[
\begin{pmatrix}
1 & 0 & 0 \\
0 & 1 & 2\pi {\rm i} k \\
0 & 0 & 1
\end{pmatrix},
\]
so $G$ is not isomorphic as a complex Lie group to any complex linear algebraic group.
Similarly, $H$ is discrete and infinite, so not a linear algebraic group.
As $G$ is connected, $G$ is the only finite index subgroup of~$G$.
\end{Example}

\begin{Lemma}
Suppose that $(X, G)$ is an effective connected complex homogeneous space $X=G/H$.
Suppose that $G$ has finitely many components and $H$ is isomorphic as a complex Lie group to a complex linear algebraic group.
Then $(X, G)$ is infinitesimally algebraic.
\end{Lemma}

\begin{proof}
We can assume that $G$ is connected, because $G$ has finitely many components.
We invoke the usual semisimple Levi quotient \cite[p. 20]{Vinberg:1994}.

Every complex linear algebraic group has an algebraic Levi decomposition $H = \red{H} \ltimes \urad{H}$ into a maximal reductive and its unipotent radical \cite[Theorem 4.3]{Hochschild:1981}, and so has a reductive quotient (over all $H$-modules simultaneously).
Take an $H$-invariant filtration \[0 = F_0 \subset F_1 \subset \cdots
 \subset F_N = \LieG/\LieH\] so that $\urad{H}$ acts trivially on the associated graded.
Each $F_i$ is $H$-invariant, so $\red{H}$-invariant, so $F_i$ has a $\red{H}$-invariant splitting $F_i=F_{i-1} \oplus F_{i-1}^{\perp}$.
We claim that $\red{H}$ acts faithfully on the associated graded.
Suppose that some element $g \in \red{H}$ acts trivially on the associated graded.
Note that~$g$ acts trivially on $F_1$, since $F_1$ is a subspace of the associated graded.
Suppose by induction that~$g$ acts trivially on $F_{i-1}$.
If $v \in F_i$, split as $v=u+u^{\perp}\in F_{i-1}\oplus F_{i-1}^{\perp}$.
So
\begin{align*}
gv-v
&=
g\bigl(u+u^{\perp}\bigr)-u-u^{\perp}=
gu^{\perp}-u^{\perp}.
\end{align*}
But $gv-v \in F_{i-1}$ so $gu^{\perp}-u^{\perp} \in F_{i-1}$.
But $u^{\perp} \in F_{i-1}^{\perp}$ which is $\red{H}$-invariant, so $gu^{\perp}-u^{\perp} \in F_{i-1} \cap F_{i-1}^{\perp}=0$.
So $gv = v$.
Every reductive complex linear subgroup of $H$ acts faithfully on~$\LieG/\LieH$ \cite[Lemma~6.1]{McKay2013}, so $g = I$.
\end{proof}

\section{Review of stable vector bundles}\label{section:stability}

In this section, we recall notions of stability for holomorphic sheaves and for holomorphic principal bundles.
All of this material is discussed in detail, for algebraic varieties, in the standard textbook \cite[Chapter 1]{Huybrechts/Lehn:2010}, and for compact K\"ahler manifolds in the standard textbook \cite[Chapter V, Section 7]{Kobayashi:1987}.
For an elegant introduction to the concepts, see \cite[Chapter 7]{Atiyah.Bott:1983}. We will need a notion of stability to relate the Cartan connection of a Cartan
geometry to the existence of a~holomorphic connection on related principal bundles. It will be these
related principal bundles to which we can apply the Shafarevich fibration.

On any compact complex manifold $M$ with K\"ahler form $\omega$, each torsion free coherent sheaf~$V$
has \emph{degree} \cite[p.~168]{Kobayashi:1987}
\[
\deg{V} := \int_M c_1(V) \wedge \omega^{-1 + \dimC{M}},
\]
and \emph{slope} $\frac{\deg{V}}{\rank{V}}$. A torsion free coherent sheaf is
\begin{itemize}\itemsep=0pt
\item
\emph{stable} if every torsion free coherent subsheaf of lower rank has lower slope,
\item
\emph{semistable} if no torsion free coherent subsheaf of lower rank has greater slope,
\item
\emph{polystable} if semistable and a direct sum of stable sheaves (equivalently, a direct sum of
stable sheaves of same slope),
\item
\emph{pseudostable} if it admits a filtration
\[0 = V_0 \subset V_1 \subset \cdots \subset V_n = V\] of holomorphic vector
subbundles, so that all of $V_1/V_0, V_2/V_1, \dots$ are polystable of equal slope \cite[p. 168]{Kobayashi:1987}.
\end{itemize}
Suppose that $G$ is a complex Lie group with a semisimple Levi quotient, i.e., some finite index
subgroup $G_0 \subset G$ has a morphism of complex Lie groups
$1 \to G_{\rm u} \to G_0 \to G_{\rm ss} \to 1$ to a complex semisimple Lie group.
We can arrange that $G_{\rm ss}$ is connected, by replacing $G_0$ with a smaller but still finite index
subgroup of $G$.
Take a holomorphic principal right $G$-bundle $G \to P \to M$ on a compact K\"ahler manifold $M$.
Let $M_0 := P/G_0$ and let $P_{\rm ss} := \amal{P}{G_0}{G_{\rm ss}}$.
The bundle $G \to P \to M$ is \emph{pseudostable} if the associated holomorphic
vector bundle $\amal{P_{\rm ss}}{G_{\rm ss}}{\LieG_{\rm ss}} \to M_0$ is pseudostable.

\section{Bundles on tori}\label{section:bundles.on.tori}
In this section, we consider how to relate holomorphic connections on different principal bundles.
We need this because the standard theory of the Shafarevich fibration (see the standard reference~\cite{Zuo:1999}) is not developed for general principal bundles, but only those with particular structure groups.
\begin{Lemma}[{\cite[Theorem 4.1]{Biswas/Gomez:2008}}]\label{lemma:BG}
Take a holomorphic principal bundle $G \to P \to T$ on a~complex torus $T$, with structure group $G$.
Suppose that some finite index subgroup $G_0 \subset G$ has semisimple Levi quotient $1 \to G_{\rm s}
 \to G_0 \to G_{\rm ss} \to 1$.
The following are equivalent:
\begin{enumerate}\itemsep=0pt
\item[$(1)$]
the bundle admits a holomorphic connection,
\item[$(2)$]
the bundle is pseudostable with vanishing first and second Chern classes,
\item[$(3)$]
after perhaps lifting to a finite unramified covering torus, the associated bundle $G_{\rm ss}$ admits a holomorphic reduction of structure group to the Borel subgroup of $G_{\rm ss}$, so that the line bundle associated to any character of the Borel subgroup has vanishing first Chern class,
\item[$(4)$]
\begin{itemize}\itemsep=0pt
\item
the bundle is pseudostable,
\item
the bundle admits a unique holomorphic flat connection,
\item
every holomorphic section of any associated vector bundle is parallel for that holomorphic flat connection,
\item
every holomorphic connection on the bundle is parallel in that holomorphic flat connection.
\end{itemize}
\end{enumerate}
\end{Lemma}

\section{Cartan connections on torus fibrations}\label{section:Cartan.torus.fib}

In this section, we extend our theory of principal bundles on tori to principal bundles on torus fibrations as they occur in Cartan geometries.
\begin{Proposition}\label{proposition:torus.fibration}
Take a holomorphic torus fibration $M \to S$ over a complex manifold~$S$.
Take a holomorphic Cartan geometry $H \to E \to M$ with infinitesimally algebraic model $(X, G)$.
Then the quotient bundle $\red{H}\to E/\urad{H} \to M$ is the pullback of a holomorphic principal bundle $\red{H} \to \red{E} \to S$.
On each fiber of $M \to S$, the bundle $E$ pulls back to have a holomorphic flat connection with holonomy in $H_{\rm u}$, and this connection varies holomorphically with the fiber.
\end{Proposition}

\begin{proof}
Take an $H$-invariant filtration $0 = F_0 \subset F_1 \subset \dots \subset F_N = \LieG/\LieH$ so
that $\urad{H}$ acts trivially on the associated graded $W=\graded{*}{F}$.
Then $Z := \GL{W}/H = \GL{W}/\red{H}$ is a quasi-projective homogeneous variety
\cite[Theorem 12.2.1]{Springer:2009}, and indeed affine; see Mumford et al.\
\cite[Theorem~1.1]{Mumford/Fogarty/Kirwan:1994} and Procesi \cite[Theorem 2]{Procesi:2007}.

Denote the fibration $s \colon M \to S$.
For each $s_0 \in S$, let $M_{s_0} := s^{-1}\{s_0\}$.
For each point~${s_0 \in S}$, let $V_{s_0}$ be the vector space of translations of the fiber $M_{s_0}$, and let $\Lambda_{s_0}
 \subset V_{s_0}$ be the discrete subgroup of translations acting trivially on $M_{s_0}$, making a vector bundle $V \to S$ and a covering space $\Lambda \to S$ with $\Lambda \subset V$, so that $M=V/\Lambda$.

Since the problem is local on $S$, we can assume that there is a holomorphic section $m \colon S \to M$ of the
bundle $s \colon M \to S$.
We can assume that $S$ is an open subset of $\C^{n}$ with local holomorphic coordinates $z^i$, and take local sections $e_1, e_2,
 \dots, e_p$ of $\Lambda \to S$ so that they form a basis of local sections of $V$.
Write any element of the total space of $V$ as $w = w^Ae_A$.
On the total space of $V$, we have local holomorphic coordinates $z^i$, $w^A$.
Since $M=V/\Lambda$ is covered by~$V$, these are also local holomorphic coordinates on $M$.
Pull back the $1$-forms ${\rm d}z^i$, ${\rm d}w^A$ to $E$, where they are a basis of the semibasic $1$-forms.
The $1$-form $\omega+\LieH \in \nForms{1}{E} \otimes^H\!(\LieG/\LieH)$ is semibasic, so $\omega+\LieH = f_i {\rm d}z^i + f_A {\rm d}w^A$
for unique holomorphic functions $f_i, f_A \colon E \to \LieG/\LieH$.
Since $\omega$ is a linear isomorphism on each tangent space of $E$, $\omega+\LieH$ is an isomorphism of the pullback tangent bundle to $\LieG/\LieH$, i.e., if we identify the $1$-forms ${\rm d}z^i$, ${\rm d}w^A$ with $\LieG/\LieH$ by any fixed linear isomorphism, then\looseness=-1
\[
f := (f_i \ f_A) \colon\ E \to \GL{\LieG/\LieH}
\]
is an $H$-equivariant holomorphic map $r_h^*f = \Ad_h^{-1} \circ f$, $h \in H$.
Hence the quotient object
\[
\overline{f} \colon\ M = E/H \to Z = \GL{W}/\red{H}
\]
is a holomorphic map to an affine variety.
\samepage{So $\overline{f}$ is constant on each torus fiber, i.e., drops to a~map $\overline{f} \colon\ S \to Z$ giving maps
\[
\begin{tikzcd}
E \arrow{d} \arrow{r}{f} & \GL{W} \arrow{d} \\
S \arrow{r}{\overline{f}} & Z
\end{tikzcd}
\]
to identify $\red{E}=E/H_{\rm u}=s^* \GL{W}$, i.e., $\red{E}$ is pulled back over $S$.}

By taking a local reduction of structure group of $\red{E} \to S$ (over an open subset of $S$, which we can assume is $S$) we can reduce the structure group of $E \to M$ over that open subset of~$S$ to a holomorphic principal $\urad{H}$-bundle $\urad{H} \to \urad{E} \to M$.
By definition of an infinitesimally algebraic homogeneous space, $\urad{H}$ has a complex Lie group morphism $\urad{H} \to \urad{\overline{H}}$ to a~unipotent complex linear algebraic group, injective on Lie algebras.
Replacing $M$ by a~finite unramified covering space, we can assume that $\urad{\overline{H}}$ is connected.
Let ${\urad{\overline{E}} := \amal{\urad{E}}{\urad{H}}{\urad{\overline{H}}}}$.
The Atiyah classes of the bundles $\urad{E}$ and $\urad{\overline{E}}$ are identified, as the Lie algebras of the groups are identified.
On a complex torus, any holomorphic principal bundle with connected unipotent complex linear algebraic structure group has a unique holomorphic flat connection by Lemma~\ref{lemma:BG}.
Such a~connection is associated to a representation of the fundamental group, and so varies holomorphically.
\end{proof}

In cohomology, if we let $V^{p,q}_{s_0} = \cohomology{p,q}{M_{s_0}}$, i.e., $V^{1,0}$ is the complex dual of $V$, then $\Lambda^* \subset V^{1,0} \oplus V^{0,1}$ is a holomorphic variation of Hodge structure \cite[p. 249]{Voisin:2007a}.

If we have a holomorphic section $m \colon S \to M$, we get
\[
\begin{tikzcd}
0 \arrow{r} & V \arrow{r} & m^*TM \arrow{r} & TS \arrow{r} & 0.
\end{tikzcd}
\]
This sequence splits by taking the map $T_{s_0} S \longrightarrow T_{m(s_0)} M$ defined by
$v \longmapsto m'(s_0)v$, so
\[
m^* TM = TS \oplus V.
\]
Hence on Atiyah classes $a(m^*TM) = a(TS)+a(V)$.

\section{Large fundamental groups}\label{section:Large.fundamental.groups}
In this section, we begin the study of fundamental groups of varieties that develop to the model of a Cartan geometry.
We build up this theory until we find that if the fibers of a fibration in a~Cartan geometry have solvable holonomy for the Cartan connection, then they are tori.

A connected complex projective variety $M$ has \emph{large fundamental group} if every connected positive-dimensional closed
complex subvariety $Z \subset M$ with normalization $Z_0 \to Z$ has infinite fundamental group image $\fundamentalGroup{Z_0} \to \fundamentalGroup{M}$ \cite{Kollar:1993}.

\begin{Lemma}\label{lemma:dev.large}
Suppose that $M$ is a smooth projective variety admitting a minimal geometry modelled on a complex homogeneous space $(X, G)$.
Suppose that \rationalMap{Z}{M} is a meromorphic map from a reduced connected compact complex space, and that this map develops into $X$.
Then the map extends to a holomorphic map $Z \to M$ lying in a leaf of the ant foliation on $M$ and the developing map takes $Z$ to a fiber of the ant fibration.
Suppose that $Z \to M$ is a finite map.
Then $Z$ has large fundamental group.
\end{Lemma}

\begin{proof}
Pullback the bundles $E|_Z = G|_Z$, and so identify the tangent bundles $ TX|_Z= TM|_Z$.
Since the model $X$ is homogeneous, its tangent bundle is spanned by global sections, and so its anticanonical bundle as well, and so nef.
The canonical bundle of $M$ is nef.
Therefore, the canonical bundle of $\cb{M}$ restricts to a nef line bundle on $Z$ with nef dual bundle the restriction of $\acb{X}$.
Hence $\cb{M}$ restricts to a trivial bundle on $Z$ as does $\cb{X}$.
Wedging sections, $ TX|_Z$ is spanned by global sections, but global sections of $ TX|_Z= TM|_Z$ can not vanish anywhere, so~${TX|_Z= TM|_Z}$ is a trivial bundle.
If a section of $TX$ vanishes at some point of $X$, then by homogeneity some such section vanishes at any point of $X$, in particular at some point of $Z$.
Any section of $TX$ which vanishes at some point of $Z$ vanishes at all points of $Z$, since $\acb{X}$ is trivial on $Z$.
Write the map $Z \to M$ as $f$, and its development $Z \to X$ as $\widehat{f}$.
So at every point of $\widehat{f}(Z)$, precisely the same vector fields on $X$ vanish and precisely the same sections of~$\acb{X}$.
In other words, $Z \to X$ lies in an anticanonical fiber.
All holomorphic functions on that fiber pull back to constants on $Z$, so $Z \to X$ lies in an ant fiber.
So $Z \to M$ lies in an ant leaf.
Each ant fiber $(X_0, G_0)$ is a connected complex homogeneous space, and its universal covering space is a Stein complex Lie group by
Lemma~\ref{lemma:ant.Stein}.
So $Z \to X$ lifts to a holomorphic map to that Stein complex Lie group from some covering space of $Z$.

Suppose that $g \colon Y \to Z$ is a holomorphic map from a connected compact normal complex space.
Under the group morphism $\fundamentalGroup{Y} \to \fundamentalGroup{Z}$, suppose further that the preimage of~$\fundamentalGroup{\widehat{Z}}$ is a finite index subgroup of $\fundamentalGroup{Y}$.
After replacing $Y$ by a finite unramified covering space, we can assume that this image is trivial.
The map $Y \to Z \to X_0$ lifts to $Y \to \widehat{X}_0$, which is a~Stein manifold.
Hence $Y \to \widehat{X}_0$ is constant and so $Y \to X_0$ and $Y \to M$ are constant so~$Y$ is a point.
In particular, if $\fundamentalGroup{Y} \to \fundamentalGroup{Z}$ is finite, $Y$ is a point.
\end{proof}

\begin{Example}
A non-uniruled smooth projective variety with a holomorphic parabolic geometry has trivial ant foliation, since its model is a flag variety, hence has trivial ant fibration.
Therefore, a meromorphic map from a reduced connected compact complex space $Z$ develops to a minimal geometry if and only if that map is constant.
So it develops to an arbitrary holomorphic parabolic geometry, on a connected manifold, just when it lies in a fiber of the fibration to the associated minimal geometry.
Roughly speaking, once we drop out the rational curves, we can not develop anything.
\end{Example}

\begin{Lemma}\label{lemma:large}
Suppose that $M$ is a smooth projective variety bearing a minimal geometry modelled on a complex homogeneous space $(X, G)$.
Suppose that \rationalMap{Z}{M} is a meromorphic map from a positive-dimensional reduced connected compact complex space.
Then this meromorphic map extends uniquely to a holomorphic map.
Suppose that $Z \to M$ is a finite map.
Take the normalization $Z_0 \to Z \to M$.
Either
\begin{enumerate}\itemsep=0pt
\item[$(1)$]
the image of $\fundamentalGroup{Z_0} \to \fundamentalGroup{M}$ is infinite or
\item[$(2)$]
$Z$ has large fundamental group.
\end{enumerate}
\end{Lemma}

\begin{proof}
Meromorphic maps to $M$ from reduced compact complex spaces extend to holomorphic maps \cite[Lemma 26]{Biswas.McKay:2016}.
Replace $Z$ by its normalization.

Suppose that the image of $\fundamentalGroup{Z} \to \fundamentalGroup{M}$ is finite.
After replacing $M$ by a finite covering space, we can arrange that the image of $\fundamentalGroup{Z}
\to \fundamentalGroup{M}$ is trivial.

Suppose that $Z \to M$ is finite.
By Remmert's proper mapping theorem \cite[p.~213]{Grauert.Remmert:1984}, the image of $Z
 \to M$ is a projective subvariety;
pick a proper curve inside the image of~${Z \to M}$.
Replace $Z$ by the preimage of that curve inside $Z$.
The curvature of the Cartan geometry ${E \to M}$ vanishes on $Z$ because the curvature is a holomorphic $2$-form, so vanishes on curves.
Since $\fundamentalGroup{Z} \to \fundamentalGroup{M}$ has trivial image, the connection has no holonomy.
The Cartan connection integrates to a developing map to the model $Z \to X$.
By Lemma~\ref{lemma:dev.large}, the curve $Z$ has large fundamental group.
Now if $Z$ is not a curve, and contains a subvariety whose normalization has
finite fundamental group image in $Z$, pick a curve in that subvariety.
\end{proof}

\begin{Corollary}\label{corollary:tori}
Suppose that $M$ is a smooth projective variety bearing a minimal geometry.
Take a smooth projective variety $Z$ with trivial canonical bundle.
If there is a finite meromorphic map \rationalMap{Z}{M} then $Z$ has a finite unramified covering by an abelian variety.
\end{Corollary}

\begin{proof}
Meromorphic maps to $M$ from reduced compact complex spaces extend to holomorphic maps \cite[Lemma 26]{Biswas.McKay:2016}.
After perhaps replacing by a finite unramified covering, any smooth projective variety $Z$ with trivial canonical bundle splits into a product $Z=A \times B$ where $A$ is an abelian variety and $B$ is a simply connected Calabi--Yau manifold \cite{Bogomolov:1974}.
Replace $Z$ by $\{a_0\} \times B$ for any $a_0 \in A$ to arrange that $Z$ is simply connected, contradicting
Lemma~\ref{lemma:large} unless $B$ is a~point.
\end{proof}

\begin{Lemma}\label{lemma:solvable.nef}
Suppose that $M$ is a connected compact K\"ahler manifold bearing a minimal geometry with solvable holonomy and either $M$ has nef canonical bundle or $M$ is Moishezon.
Then $M$ is a complex torus, after perhaps replacing $M$ by a finite unramified covering.
\end{Lemma}

\begin{proof}
By Lie's theorem (see \cite[Theorem 3]{Serre:2001}), every associated vector bundle of the bundle $E_G \to M$ has a filtration into a flag of holomorphic vector bundles, with quotient line bundles.
Those line bundles admit holomorphic connections, so are nef.
In particular, the adjoint bundle~$E_G \times^G \LieG$ has such a filtration.
Therefore, the adjoint bundle is an extension of nef line bundles, and so is a nef vector bundle.
Since $TM$ is a quotient bundle of the adjoint bundle:
\[
E_G \times^G \LieG = E \times^H \LieG \to E \times^H \pr{\LieG/\LieH} = TM,
\]
$\acb{M}$ is nef.
But $\cb{M}$ is nef (by Theorem~\ref{theorem:drop} if $M$ is Moishezon), so $\cb{M} = 0$, after perhaps replacing by a finite unramified covering, hence a complex torus after perhaps another finite unramified covering \cite{Biswas.McKay:2010}.
\end{proof}

Take a complex homogeneous space \pr{X,G} and a point $x_0 \in X$.
Take the ant fibration $X = G/H \to \overline{X} = G/\overline{H}$, with fiber \pr{X_0,\GZ} where
\[
X_0 = \overline{H}/H = \GZ/\Gamma
\]
with $\GZ = \overline{H}/K$ and $\Gamma = H/K$ where
\[
K = \bigcap_{g\in\overline{H}} gHg^{-1}
\]
contains $H^0$.
Denote the universal covering group as $\tGZ \to \GZ$ and let $\widetilde{\Gamma} \subset \tGZ$ be the
preimage of $\Gamma \subset \GZ$.

Take a complex Lie algebra morphism
$
\widetilde{f} \colon \C^{n} \to \LieGZ
$
from the abelian Lie algebra $\C^{n}$ and exponentiate into a complex Lie group morphism
$
\widetilde{f} \colon \C^{n} \to \GZ.
$
Suppose that $\Lambda \subset \C^{n}$ is a~lattice and that $\widetilde{f}\Lambda \subset \widetilde{\Gamma}$.
Let $T := \C^{n}/\Lambda$.
Quotient to a map $f \colon T \to X_0$.
Take an element~${g \in \GZ}$.
The map $gf \colon T \to X$ is a \emph{hoop}.

\begin{Lemma}\label{lemma:is.a.hoop}
Every holomorphic map from a connected compact K\"ahler manifold with trivial canonical bundle to a complex homogeneous space which pulls back the canonical bundle to be trivial is a hoop.
\end{Lemma}

\begin{proof}
Take a connected compact K\"ahler manifold $Z$ with trivial canonical bundle and a holomorphic map $f \colon Z \to X$ so that $f^* \cb{X}=0$.
Clearly $f(Z)$ lies in an anticanonical fiber of~$X$, so without loss of generality we can replace $X$ by that fiber, and so assume that $X=G/\Gamma$ where $\Gamma \subset G$ is a discrete subgroup.
Since $Z$ is compact, every holomorphic function on $X$ pulls back to a constant on $Z$, so $Z$ lies in an ant fiber of $X$.
Again, without loss of generality, replace $X$ by that fiber.
After perhaps replacing $Z$ by a finite unramified covering space, split~$Z = T \times B$ where $T$ is a complex torus and $B$ is a simply connected compact K\"ahler manifold with trivial canonical bundle.
For each $t_0 \in T$ the map $f$ lifts to take $\{t_0\} \times B$ to~$\widetilde{X}_0$, a simply connected complex Lie group, so $B$ is a point.
The map $f \colon T \to X$ pulls back the right invariant Maurer--Cartan form ${\rm d}g g^{-1}$ to be $f^* {\rm d}g g^{-1}
 = c {\rm d}z$, a constant multiple of the translation invariant coframing.
The holomorphic vector fields on $T$, all of which are constant, map by $f'$ to holomorphic sections of the pullback tangent bundle (also constant), representing tangent vector fields, taking brackets (all zero) to brackets, so $f$ is a morphism of Lie algebras from an abelian Lie algebra.
Therefore, the lift $\widetilde{f} \colon \C^{n} \to G$ which satisfies $\widetilde{f}(0)=1$ (after suitable translation) is a complex Lie group morphism.
\end{proof}
\begin{Corollary}\label{corollary:hoops}
Suppose that $M$ is a non-uniruled smooth projective variety bearing a holomorphic Cartan geometry.
Take a finite map $Z \to M$ which develops to $X$, from a smooth projective variety $Z$ with trivial canonical bundle.
Then, up to replacing $Z$ by a finite unramified covering map, $Z$ is an abelian variety, and the developing map is a hoop.
\end{Corollary}

\begin{Example}
If $M = \C^{n}/\Lambda_0$ is a complex torus and $X = \C^{n}/\Lambda_1$ is another complex torus and $G = X$, then $M$ has the obvious flat \pr{X,G}-geometry, unique up to linear transformation, with identity map as developing map.
If, after such a linear transformation, we can find a linear subspace $V \subset \C^{n}$ containing a
lattice $\Lambda \subset \Lambda_0 \cap \Lambda_1$, then $V/\Lambda$ develops to the model by the identity map.
\end{Example}

\begin{Lemma}\label{lemma:abelian.pancake}
Take a non-uniruled smooth projective variety $M$ with a holomorphic Cartan geometry $E \to M$ with model $(X, G)$.
Take a rational dominant map $\rationalMap{M}{S}$.
Suppose that on the general fiber of that map, the pullback bundle of $E_G$ admits a holomorphic connection, perhaps not equal to the Cartan connection, with holonomy acting on $\LieG$ as a solvable group of linear transformations.
Then the general fiber admits a finite unramified covering by an abelian variety.
Moreover, $\cb{M}$ is trivial on the general fiber.
If the connection is flat, then the holonomy on each fiber is abelian.
\end{Lemma}

\begin{proof}
Take a general fiber $M_{s} $ of $\rationalMap{M}{S}$ and let $E_{G_{\rm s} } := E_G|_{M_{s} }$.
By Lie's theorem \cite[Theorem 3]{Serre:2001}, the adjoint vector bundle $\ad{E_{G_{\rm s} }} := \amal{E_{G_{\rm s}}}{G}{\LieG}$ of the bundle $E_{G_{\rm s} }
\to M_{s} $ has a filtration into a flag of holomorphic vector bundles, with quotients being line bundles.
The holomorphic connection induces a holomorphic connection on each quotient line bundle, so each quotient line bundle is nef.
Therefore, the adjoint bundle is an extension of nef line bundles, and so is a nef vector bundle.
Since the tangent bundle $TM|_{M_{s} }$ is a quotient bundle of the adjoint bundle:
\[
\amal{E_G}{G}{\LieG} = \amal{E}{H}{\LieG} \to \amal{E}{H}{\pr{\LieG/\LieH}} = TM,
\]
$\acb{M}$ is nef along $M_{s} $.
But $\cb{M}$ is nef on $M$, since $M$ is not uniruled, so also nef on $M_{s} $ and so $\cb{M} = 0$ on $M_{s} $.

The general fiber $M_{s} $ is smooth.
By adjunction, the anticanonical bundle $\acb{M_{s}} = - \cb{M} |_{M_{s}}$ is nef.
But $\cb{M_{s}}$ is also nef as every subvariety of $M$ is minimal by Theorem~\ref{theorem:drop}.
Hence the general fiber $M_{s}$ has trivial canonical bundle.
By Corollary~\ref{corollary:tori}, $M_{s}$ admits a finite unramified covering by an abelian variety, i.e., the general fiber of $\rationalMap{M}{S}$ admits a finite unramified covering by an abelian variety.
Since the fundamental group of any abelian variety is abelian, its image under the holonomy morphism is also abelian.
\end{proof}

\section{Review of the Shafarevich fibration}\label{section:Shafarevich}
In this section, we review the definition and basic properties of the Shafarevich fibration (a well known construction in algebraic geometry).

Take any compact K\"ahler manifold $M$ equipped with a morphism $\rho \colon \fundamentalGroup{M} \to G$ of groups where $G$ is a complex semisimple Lie group.
A \emph{Shafarevich map} for $\rho$ is a meromorphic map~${\Shaf[\rho]{} \colon \rationalMap{M}{\Shaf[\rho]{M}}}$ to a~normal complex space, surjective and with connected fibers, so that a~connected closed subvariety $Z \subset M$ through a very general point is mapped to a~point by the Shafarevich map just when the normalization of $Z$ has fundamental group with finite image inside $G$.
Every group morphism $\rho \colon \fundamentalGroup{M} \to G$ with Zariski dense image in a~complex semisimple Lie group $G$ has a Shafarevich map \cite[Definition~2.13]{Campana/Claudon/Eyssidieux:2015}, \cite[p.~105]{Zuo:1999} with the additional property that, after replacing $M$ by a finite covering space, the target $\Shaf[\rho]{M}$ is bimeromorphic to the total space of a smooth fibration of complex tori over an algebraic variety of general type \cite[Theorem~4]{Campana/Claudon/Eyssidieux:2015}.

Suppose that $\Sigma \subset \fundamentalGroup{M}$ is the kernel of a representation of $\fundamentalGroup{M}$.
Also after replacing $M$ by a finite covering space, there is a map, also called a Shafarevich map,
$
\Shaf[\Sigma]{} \colon \rationalMap{M}{\Shaf[\Sigma]{M}}
$
which contracts to a point precisely those subvarieties $Z \subset M$ passing through a very general point whose fundamental group has a finite index subgroup lying inside the given subgroup $\Sigma \subset \fundamentalGroup{M}$ \cite[Theorem 1.3]{Katzarkov:1999}, \cite{Katzarkov:1997}; for compact K\"ahler manifolds the existence of this Shafarevich map follows
from \cite[Theorem 4]{Campana/Claudon/Eyssidieux:2015} and \cite{Katzarkov:1997}.
A representation $\rho \colon \fundamentalGroup{M} \to G$ of the fundamental group of a compact K\"ahler manifold is \emph{big} if the associated Shafarevich map is a birational map.

Take a complex Lie group $G$, with identity component $G^0$, whose component group $G/G^0$ admits a faithful finite-dimensional representation.
Then $G$ admits a faithful finite-dimensional holomorphic representation just when $G^0$ is the semidirect product of a simply connected solvable Lie group and a connected reductive complex
Lie group \cite[Theorem~16.3.7]{Hilgert.Neeb:2012}.

Take a compact irreducible reduced complex space $M$ containing no rational curves.
Suppose that $X$ is a compact irreducible reduced complex space. Then every meromorphic map $\rationalMap{X}{M}$ extends to a unique holomorphic map $X \to M$, by Hironaka's Chow lemma \cite[Corollary~2.9]{Grauert:1994}.

The \emph{Iitaka conjecture}, also called the \emph{$C_{mn}$ conjecture}, for a dominant rational morphism $\rationalMap{X}{Y}$ of projective varieties with connected fibers, claims that $\kappa_X \ge \kappa_Y + \kappa_{X_y}$ for the general fiber $X_y$.
The Iitaka conjecture is verified for surjective morphisms for which all sufficiently high powers of the canonical bundle of the general fiber are spanned by global sections \cite[Corollary~1.2]{Kawamata:1985}.

\section{Applying the Shafarevich fibration}
In this section, the most difficult and important section, we prove that, in any Cartan geometry, the Shafarevich fibration is globally defined, not just a meromorphic map.
\begin{Example}
If a Cartan geometry $H \to E \to M$ has induced bundle $G \to E_G
\to M$ which is pseudostable for some Higgs field, then $G \to E_G \to M$ admits a
flat holomorphic connection \cite[Theorem~1.1]{Biswas/Gomez:2008}.
The flat connection might not equal the Cartan connection.
\end{Example}

\begin{Theorem}\label{theorem:JR}
On a connected smooth projective variety $M$, take a holomorphic Cartan geometry $H\to E \to M$ modelled on an infinitesimally algebraic complex homogeneous space~\pr{X,G}.
Suppose that the induced $G$-bundle $G \to E_G \to M$ also admits a flat holomorphic connection.
Then $M$ belongs to a tower of holomorphic fibrations $M \to M' \to S$ where~$M'$ has no rational curves.
The Cartan geometry is lifted from~$M'$.
The map $M' \to S$ is a holomorphic fibration of abelian varieties.
On every fiber, the Cartan connection is a holomorphic connection on the pullback $G$-bundle with solvable holonomy.
The base $S$ has ample canonical bundle.
\end{Theorem}

\begin{proof}
Replace $M$ by $M'$ without loss of generality, so assume that $M$ contains no rational curves.
Replace $M$ by a finite unramified covering space, so that we can replace $G$ by a finite index subgroup,
and therefore assume that $G$ has a semisimple Levi quotient over $\LieG$
\[
1 \to G_{\rm s} \to G \to G_{\rm ss} \to 1.
\]
The flat holomorphic connection induces a flat holomorphic connection on the quotient $G_{\rm ss} \to E_{G_{\rm ss}} :=
 E/G_{\rm s}\to M$.
Let $L \subset G_{\rm ss}$ be the complex algebraic Zariski closure of the holonomy group of that connection.
This connection induces a flat holomorphic connection on a principal subbundle $L\to E_L \to M$,
$E_L \subset E_{G_{\rm ss}}$.
It therefore also induces a flat holomorphic connection on the quotient bundle $L_{\rm ss} \to E_{L_{\rm ss}}
 = E_L/\sqrt{L} \to M$, where $1 \to \sqrt{L} \to L \longrightarrow
L_{\rm ss} \to 1$ is the semisimple Levi quotient \cite[Theorem 4.3]{Hochschild:1981}.
Denote by~${\rho \colon \fundamentalGroup{M} \to L_{\rm ss}}$ the induced representation.

After perhaps replacing $M$ by a finite unramified covering space, the representation $\rho \colon
\fundamentalGroup{M}\allowbreak\to L_{\rm ss}$ lifts from a big representation $\fundamentalGroup{S} \to L_{\rm ss}$,
where $S := \Shaf[\rho]{M}$ and $S$ is a projective variety of
general type \cite[Theorem 5 (and remarks following)]{Zuo:1999}.
The Shafarevich fibration~${\rationalMap{M}{S}}$ is defined and proper on a Zariski dense open subset of $M$
\cite[Theorem 4.1]{Kollar:1993}.
The general fiber has finite unramified covering on which the holonomy of $E_{L_{\rm ss}}\to M$ is trivial, i.e., the holonomy of $E_L \to M$ lies in $\sqrt{L}$, i.e., the holonomy of the $E_G$-connection lies in the preimage of $\sqrt{L}$ in $G$, i.e., an extension of $\sqrt{L}$ by $G_{\rm s}$.
By Lemma~\ref{lemma:abelian.pancake}, the general fiber admits a~finite unramified covering by an abelian variety, with abelian holonomy image and $\cb{M}$ is trivial on the general fiber.

The general fiber of $\rationalMap{M}{S}$ has a finite unramified covering by an abelian variety, so the Iitaka conjecture is verified for $\rationalMap{M}{S}$: $\kappa_M \ge \kappa_S = \dim{S}$.
Since $\kappa_M \ge \dim{S}\ge 0$, we know that $M$ has an Iitaka fibration.
The general fiber $M_{s}$ of $\rationalMap{M}{S}$ admits a finite unramified covering by an abelian variety and $\cb{M}$ is trivial on $M_{s}$, so $M_{s}$ lies in a fiber of the Iitaka fibration, giving a rational map $\rationalMap{S}{\Ii{M}}$ making a commutative diagram
\[
\begin{tikzcd}
M \arrow[<->]{r}\arrow[->, dashed,densely dashed]{d} & M \arrow[->, dashed,densely dashed]{d} \\
S \arrow[->,dashed,densely dashed]{r} & \Ii{M}.
\end{tikzcd}
\]
The rational map $\rationalMap{S}{\Ii{M}}$ is dominant, as the top row of the commutative diagram is dominant.
By dominance, $\dim{S} \ge \dim{\Ii{M}} = \kappa_M \ge \kappa_S = \dim S$.
Consequently $\rationalMap{S}{\Ii{M}}$ is a~birational map identifying $\rationalMap{M}{S}$ with the Iitaka fibration
\cite[Theorem 10.6]{Iitaka:1982}, i.e., the Shafarevich fibration is a model of the Iitaka fibration.

Since the generic fiber of the Iitaka fibration has canonical bundle spanned by global sections, so does the total
space \cite[Theorem 4.4]{Lai:2011}, i.e., the canonical bundle of $M$ is spanned by global sections.
Hence the Iitaka fibration is holomorphic, i.e., we can arrange that $M \to S$ is holomorphic.

The canonical bundle of $M$ restricts to a nef line bundle on every fiber of $M \to S$, since~$M$ has canonical bundle spanned by global sections.
Some power of the canonical bundle $K_M$ is pulled back to $M$ from $M \to S$ \cite[Theorem 2.1.26]{Lazarsfeld:2004}.
Hence that power is trivial on the fibers, and so $\acb{M}$ is also nef on every fiber of $M \to S$.
Therefore, all fibers of $M \to S$ have the same dimension \cite[Theorem 2]{Kawamata:1991}.

Take any resolution of singularities $S_0 \to S$ and pullback $M$ to get a dominant morphism~${M_0 \to S_0}$
birational to $M \to S$ with smooth $M_0$ and $S_0$.
After replacing $M$ with a~finite unramified covering space, there is some resolution of singularities $S_0 \longrightarrow
S$ over which~${M \to S}$ is birational to a holomorphic fibration (indeed, a smooth morphism, which is also a $C^{\infty}$ fiber bundle \cite[Definition 5.7]{Kollar:1993}) with fibers abelian varieties, with smooth total space, over a smooth variety of general type, equipped with a holomorphic
section \cite[Theorem~6.3]{Kollar:1993}:
\[
\begin{tikzcd}
\dot{M} \arrow[->]{d} \arrow[<->, dashed,densely dashed]{r} & M_0 \arrow[->]{r}\arrow[->]{d} & M \arrow[->]{d}\\
\dot{S} \arrow[<->, dashed,densely dashed]{r} & S_0 \arrow[->]{r} & S.
\end{tikzcd}
\]
Rational maps to $M$ extend to holomorphic maps, since $M$ contains no rational curves.
Extend the rational map $\rationalMap{\dot{M}}{M}$ to a holomorphic map $\dot{M} \to M$.
Up to birational isomorphism, there is a section $\dot{S} \longrightarrow
 \dot{M}$ \cite[Theorem 6.3]{Kollar:1993}, and so a rational section $\rationalMap{S}{M}$.
Extend the rational section to a holomorphic map $S \to M$, and hence a holomorphic section.
This section maps curves in $S$ to isomorphic curves in $M$; but $M$ contains no rational curves, and so $S$ contains no rational curves.
Hence meromorphic maps to $S$ extend to be holomorphic.
Holomorphically extend $\rationalMap{\dot{S}}{S}$ and $\rationalMap{M}{S}$:
\[
\begin{tikzcd}
\dot{M} \arrow[->,shift left]{r}\arrow[->]{d} & M \arrow[->]{d} \arrow[->, dashed,densely dashed,shift left]{l} \\
\dot{S} \arrow[->,shift left]{r} & S\arrow[->, dashed,densely dashed,shift left]{l}.
\end{tikzcd}
\]

By Zariski's main theorem (see \cite[Corollary 11.4]{Hartshorne:1977}), the map $\dot{M} \to M$ has connected fibers, as does $\dot{S} \to S$.
Hence $\dot{M} \to \dot{S}$ is also an Iitaka fibration.
Each fiber of $\dot{M}\to \dot{S}$ is an abelian variety mapping holomorphically to a fiber of $M \to S$.

On the generic fiber $\dot{M}_{\dot{s}}$, this map $\dot{M}_{\dot{s}}\to M_{s}$ is birational with exceptional locus covered in rational curves.
But neither fiber contains rational curves.
Therefore, the generic fiber $\dot{M}_{\dot{s}}$ of~$\dot{M}\to \dot{S}$ is mapped isomorphically to a fiber
of $M\to S$.

Suitable powers of the canonical bundles of $\dot{M}$ and $M$ are spanned by global sections, pulled back from sections of line bundles on $\dot{S}$ and $S$ respectively.
All sections on $M$ and on $\dot{M}$ are pulled back from $S$ and $\dot{S}$ respectively.
So the vanishing locus of any such section is a union of fibers.

The holomorphic map $\dot{M} \to M$ is an isomorphism except on the ramification divisor, an effective divisor given by the vanishing of the determinant of the derivative of $\dot{M} \to M$.
Applying the derivative of $\dot{M} \to M$ to sections of the canonical bundle on $M$, we get sections of the canonical bundle on $\dot{M}$, with divisor the sum of the pullback canonical divisor and the ramification divisor.
But the sections vanish on unions of fibers.
Hence the ramification divisor of $\dot{M}\to M$ is the preimage of some effective divisor on $\dot{S}$.

Away from the image in $S$ of the support of that divisor in $\dot{S}$, our section of $M \to S$ strikes a smooth fiber in a single smooth point, since $\dot{M} \to M$ is biholomorphic on those fibers.
All fibers have intersection number 1 with the section, and so all fibers are smooth near the image of the section, and the image of the section is smooth.
The variety $S$ is smooth, since the image of the section is isomorphic to $S$.

Since the fibers of $M \to S$ are all of the same dimension, at any point of any fiber at which the reduction of the fiber is smooth, we can pick a transverse complex submanifold, a~local holomorphic section of $M \to S$.
All nearby fibers have intersection number 1 with the local section, and so with our fiber: the generic point of every fiber of $M
\to S$ is reduced.

If some fiber $M_{s}$ of $M \longrightarrow
S$ is reducible, its preimage in $\dot{M}$ is reducible.
As the fibers of~${\dot{M} \longrightarrow
M}$ are connected, the preimage in $\dot{M}$ of $M_{s}$ is a connected reducible subvariety.
Since the fibers of $\dot{M} \to \dot{S}$ are irreducible, the preimage of $M_{s}$ in $\dot{M}$ is a union of fibers over a~connected reducible subvariety of $\dot{S}$.
This subvariety of $\dot{S}$ maps to a point in $S$.
Because the subvariety is connected, all of the fibers over the subvariety map to the same component of~$M_{s}$.
But $\dot{M} \to M$ is surjective, since it is holomorphic and birational.
Therefore, all fibers of~$M \to S$ are irreducible.

The fibers $\dot{M}_{\dot{s}}$ of $\dot{M} \to \dot{S}$ all have the same images in homology inside $M$, all nonzero as they are algebraic cycles.
So the derivative of the holomorphic map $\dot{M}_{\dot{s}} \to M$ drops rank on a ramification divisor.
All holomorphic sections of a suitable positive power of the canonical bundle of $M$ are pulled back from sections of a suitable line bundle on $S$.
Find a section of that suitable line bundle, nonzero at $s$, and pullback to get a section of $\cb{M}$ nowhere zero near $M_{s}$.
Take a local holomorphic section of the canonical bundle of $S$ defined and nonzero at $s$, and divide the two sections to get a holomorphic section of the relative canonical bundle of $M \to S$ defined near $M_{s}$.
Plugging the derivative of $\dot{M}_{\dot{s}} \to M$ into this section gives a section of the canonical bundle of the abelian variety $\dot{M}_{\dot{s}}$, not vanishing somewhere and therefore not vanishing anywhere as the canonical bundle of an abelian variety is trivial.
Therefore, $\dot{M}_{\dot{s}} \to M$ is a~local biholomorphism on every fiber, taking each fiber $\dot{M}_{\dot{s}}$ to a fiber $M_{s}$.
The map $\dot{M} \to M$ is birational, so the map $\dot{M}_{\dot{s}} \to M_{s}$ on each fiber is a biholomorphism to its image.

Take a local holomorphic section of $M \longrightarrow
 S$ transverse to the image of a fiber $\dot{M}_{\dot{s}}$.
The section strikes the generic nearby fiber of $M \to S$ in a single smooth point, since $\dot{M} \longrightarrow
M$ is biholomorphic on generic fibers.
Generic fibers have intersection number 1 with the section, and so all fibers have intersection 1 with the section, so are smooth and reduced near the image of the section.
Hence all fibers of $M \longrightarrow
S$ are everywhere smooth and reduced and irreducible, and so the maps $\dot{M}_{\dot{s}} \to M_{s}$ are all isomorphisms.

Since $S$ has big canonical bundle, its Iitaka fibration $S \longrightarrow
 \Ii{S}$ is birational, with exceptional fibers covered by rational curves \cite{Kawamata:1991}.
But $S$ contains no rational curves, so there are no exceptional fibers.
Since rational maps to $S$ extend to become holomorphic, the Iitaka fibration~${S \longrightarrow
 \Ii{S}}$ has a holomorphic inverse.
The exceptional locus of the canonical morphism is then also covered by rational curves, so is empty, so $S$ has ample canonical bundle.
\end{proof}

\section{Finite holonomy}

\begin{Example}\label{example:finite.holonomy}
Suppose that \pr{X',G} is a complex homogeneous space and that $X'$ is a complex torus $X' = \C^{n}/\Lambda$.
Take a finite index subgroup $\Lambda_0 \subset \Lambda$.
Let $M' := \C^{n}/\Lambda_0 \to X' = \C^{n}/\Lambda$ be the obvious projection, and pullback the geometry to $M'$, i.e., let $E \longrightarrow
 M'$ be the pullback $H$-bundle
\[
\begin{tikzcd}
E \arrow{r}\arrow{d} & G \arrow{d} \\
M' \arrow{r} & X'.
\end{tikzcd}
\]
The map $M' \longrightarrow
 X'$ is the developing map of a flat holomorphic Cartan geometry modelled on~$(X', G)$: the pullback of the model geometry.
The holonomy morphism has finite image in $G$.
Suppose that $X \to X'$ is a $G$-equivariant holomorphic map of complex homogeneous $G$-spaces, a complex
homogeneous bundle over the complex torus, say $X = G/H$ and $X' = G/H'$, and suppose that $H/H'$ is a flag variety.
Then the quotient $M := E/H$ is the lift of $E \to M'$ to a~flat holomorphic $(X, G)$-geometry.
\end{Example}

\begin{Theorem}\label{theorem:finite.holonomy}
Take a smooth projective variety $M$ and a flat holomorphic Cartan geometry on $M$ modelled on a complex homogeneous space \pr{X,G}.
If the holonomy morphism of the geometry has finite image then, after perhaps replacing $M$ by a finite unramified covering space, the geometry is constructed as in the above Example~{\rm \ref{example:finite.holonomy}}.
\end{Theorem}

\begin{proof}
We can drop to assume that $M$ is not uniruled.
We can assume that the model $X$ is connected.
Since the Cartan connection is flat, we can take a developing map and associated holonomy morphism.
Lift to a finite unramified covering space to arrange trivial holonomy, so the geometry is the pullback via a developing map
$\widehat{f} \colon M \to X$.
Pulling back, ${TM \cong \widehat{f}^* TX}$.
Since the model $X$ is homogeneous, $TX$ is spanned by global sections and so the anticanonical bundle~$\acb{X}$ of $X$ is spanned by global sections.
The canonical bundle of $M$ is nef.
But~${\cb{M} = \widehat{f}^* \cb{X}}$.
Therefore, the canonical bundle of $M$ is trivial.
So the developing map has image inside an anticanonical fiber.
Since the developing map is a local biholomorphism, the anticanonical fibers are open sets in $X$.
By compactness of $M$, the developing map is a finite unramified covering map to a component of $X$.
By Lemma~\ref{lemma:is.a.hoop}, the developing map is a hoop.
So $X$ is an abelian variety~${X = \C^{n}/\Lambda}$.
The group $G$ can be any complex Lie group acting on $X$, while the effectivization is a finite extension of $X$ by some linear transformations of $\C^{n}$ preserving $\Lambda$.
The developing map is an isogeny of abelian varieties $M \to X$.
\end{proof}

\section{Which subvarieties develop to the model}
We sum up our results into a theorem from which we easily obtain Theorem~\ref{theorem:main}.
\begin{Theorem}\label{theorem:big.1}
Take a non-uniruled smooth projective variety $M$ and a flat holomorphic Cartan geometry on $M$ modelled on a complex homogeneous space \pr{X,G}, where $G$ is a complex linear algebraic group.
After perhaps replacing $M$ by a finite unramified covering space, $M$ is a~holomorphic fibration $M \to \grave{M}
 \longrightarrow
 M'$ of abelian varieties over a smooth complex projective variety $\grave{M}$, refining the fibration of Theorem~{\rm\ref{theorem:JR}}, so that a finite map $Z \longrightarrow
 M$ from a~connected projective variety develops to the model $($after perhaps replacing $Z$ by a finite unramified covering$)$ just when the image of $Z \longrightarrow
 M$ lies in a fiber of $M \longrightarrow \grave{M}$.
Every fiber of~${M \to \grave{M}}$ has a finite unramified covering space which develops to a hoop in the model.
The map $M \to \grave{M}$ has a holomorphic section.
\end{Theorem}

\begin{proof}
If the holonomy is finite, apply Theorem~\ref{theorem:finite.holonomy}.
Suppose that the holonomy is infinite.
Let $\Sigma \subset \fundamentalGroup{M}$ be the kernel of the holonomy morphism $\fundamentalGroup{M} \to G$.
Any subvariety $Z \subset M$ whose fundamental group lies in $\Sigma$ has a developing map $\widehat{Z} \longrightarrow
 X$ from its normalization.
The developing map identifies the ambient tangent bundles $ TX|_{\widehat{Z}} = TM|_{\widehat{Z}}$.
Since $M$ contains no rational curves, its canonical bundle is nef, while the tangent bundle of $X$ is spanned by global sections, so both canonical bundles restrict to be trivial along $\widehat{Z}$.
By Lemma~\ref{lemma:dev.large}, $Z \to M$ lies in a leaf of the ant foliation, while $\widehat{Z} \to X$ lies in a fiber of the ant fibration.
Moreover, $Z$ has large fundamental group.

Every smooth fiber $M_{s}$ of $\Shaf[\Sigma]{} \colon \rationalMap{M}{S := \Shaf[\Sigma]{M}}$ has a finite unramified covering which develops to the model, by definition of $\Shaf[\Sigma]{}$.
By Lemma~\ref{lemma:abelian.pancake}, $M_{s}$ has a finite covering by an abelian variety.
By Corollary~\ref{corollary:hoops}, some finite covering of $M_{s}$ maps to $X$ by a hoop.
Every fiber~$M_{s}$ develops to $X$, so if $Z$ lies inside a smooth fiber $M_{s}$, then $Z$ develops to $X$.
Let~${M \to \grave{M}}$ be~${\Shaf[\Sigma]{}\colon M \to \Shaf[\Sigma]{M}}$.
The map $\Shaf[\Sigma]{} \colon M \to \Shaf[\Sigma]{M}$ comes with a rational section,
say $\rationalMap[s]{\grave{M}}{M}$ \cite[Theorem~6.3]{Kollar:1993}.
Meromorphic maps to $M$ extend to a holomorphic maps, since $M$ contains no rational curves.
\end{proof}

\begin{Conjecture}
Theorem~{\rm \ref{theorem:main}} remains true with the words changed as
\begin{gather*}
\text{smooth projective variety}\to\text{Fujiki manifold},\\
\text{abelian variety}\to\text{complex torus},\\
\text{abelian varieties}\to\text{complex tori}.
\end{gather*}
\end{Conjecture}

\subsection*{Acknowledgements}

This research was supported in part by the International Centre for Theoretical Sciences (ICTS) during a visit for participating in the program - Analytic and Algebraic Geometry (Code: ICTS/aag2018/03).
It is a pleasure to thank the referees who made a significant contribution to clarifying the paper.
BM thanks Anca Musta\c{t}\u{a} and Andrei Musta\c{t}\u{a} for help with algebraic geometry, and Sorin Dumitrescu for help with Cartan geometries.
This article is based upon work from COST Action CaLISTA CA21109 supported by COST (European Cooperation in Science and Technology).
IB is partially supported by a J.C.~Bose Fellowship (JBR/2023/000003).


\providecommand{\MR}{\relax\ifhmode\unskip\space\fi MR }
\providecommand{\MRhref}[2]{%
 \href{http://www.ams.org/mathscinet-getitem?mr=#1}{#2}
}
\providecommand{\href}[2]{#2}

\pdfbookmark[1]{References}{ref}
\LastPageEnding

\end{document}